%% file: MSMFEquads.tex
\documentclass[11pt]{article}

\usepackage{amssymb}
\usepackage{amsmath}
\usepackage{graphicx}
\usepackage{caption}
\usepackage{subcaption}
\usepackage{float}
\usepackage{dcolumn}
\usepackage{amsthm}
\usepackage{empheq}
\usepackage{multirow}
\usepackage{enumitem}
\usepackage{amsfonts}
\usepackage{color}
\usepackage[sort]{cite}

\usepackage{pgf}
\usepackage{tikz}
\usetikzlibrary{arrows.meta,decorations.pathmorphing,backgrounds,positioning,fit}

\numberwithin{equation}{section}

\def\Line(#1,#2)(#3,#4){\qbezier(#1,#2)(#1,#2)(#3,#4)}

\newcommand{\inp}[2][]{\left(#1, #2\right)}
\newcommand{\gnp}[2][]{\langle#1, #2\rangle}

\newtheorem{remark}{Remark}[section]
\newtheorem{lemma}{Lemma}[section]
\newtheorem{corollary}{Corollary}[section]
\newtheorem{theorem}{Theorem}[section]

\def\Tc{\mathcal{T}}

\def\Qc{\mathcal{Q}}

\def\BDM{\mathcal{BDM}}
\def\BDDF{\mathcal{BDDF}}
\def\RT{\mathcal{R\!T}}

\def\Pc{\mathcal{P}}

\def\X{\mathbb{X}}
\def\W{\mathbb{W}}
\def\R{\mathbb{R}}
\def\M{\mathbb{M}}
\def\S{\mathbb{S}}
\def\N{\mathbb{N}}

\def\x{\mathbf{x}}

\def\s{\sigma}
\def\t{\tau}
\def\g{\gamma}

\def\tet{\theta}
\def\del{\delta}

\def\r{\mathbf{r}}

\def\As{A_{\s\s}}
\def\Au{A_{\s u}}
\def\Ag{A_{\s\g}}

\def\Eh{\hat{E}}
\def\Ah{\hat{A}}

\def\rh{\hat{\r}}
\def\eh{\hat{e}}

\def\nh{\hat{n}}

\def\xh{\hat{x}}
\def\yh{\hat{y}}
\def\sh{\hat{\s}}
\def\th{\hat{\t}}

\def\ch{\hat{\chi}}
\def\teth{\hat{\tet}}
\def\delh{\hat{\del}}
\def\Xh{\hat{\X}}
\def\Vh{\hat{V}}

\def\tr{\operatorname{tr\,}}

\def\skew{\operatorname{Skew}}

\def\curl{\operatorname{curl}}
\def\dvr{\operatorname{div}}          
\def\dvrg{\operatorname{div}}  

\def\DF{D\!F}

\def\O{\Omega}
\def\Gn{\Gamma_N}
\def\Gd{\Gamma_D}

\usepackage{geometry}
\begin{document}

\title{A multipoint stress mixed finite element method for elasticity on quadrilateral grids}

\author{Ilona Ambartsumyan\thanks{Department of Mathematics, University of
		Pittsburgh, Pittsburgh, PA 15260, USA;~{\tt \{ila6@pitt.edu, elk58@pitt.edu, yotov@math.pitt.edu\}}. Partially supported by DOE grant DE-FG02-04ER25618 and NSF grants DMS 1418947 and DMS 1818775.} \thanks{Institute for Computational Engineering and Sciences, The University of Texas at Austin, Austin, TX 78712, USA;
{\tt \{ailona@austin.utexas.edu, ekhattatov@austin.utexas.edu\}}.}~\and
		Eldar Khattatov\footnotemark[1] \footnotemark[2] \and
		Jan Nordbotten\footnotemark[3]\thanks{Department of Mathematics, University of Bergen, Bergen, 7803, Norway;~{\tt \{Jan.Nordbotten@uib.no\}}.
Funded in part through Norwegian Research Council grants 250223, 233736, 228832.
}\and
		Ivan Yotov\footnotemark[1]~}

\date{\today}
\maketitle
\begin{abstract}
We develop a multipoint stress mixed finite element method for linear
elasticity with weak stress symmetry on quadrilateral grids, which can
be reduced to a symmetric and positive definite cell centered system.
The method is developed on simplicial grids in \cite{msmfe_simp}. The
method utilizes the lowest order Brezzi-Douglas-Marini finite element
spaces for the stress and the trapezoidal quadrature rule in order to
localize the interaction of degrees of freedom, which allows for local
stress elimination around each vertex. We develop two variants of the
method. The first uses a piecewise constant rotation and results in a
cell-centered system for displacement and rotation. The second uses a
continuous piecewise bilinear rotation and trapezoidal quadrature rule
for the asymmetry bilinear form. This allows for further elimination
of the rotation, resulting in a cell-centered system for the
displacement only. Stability and error analysis is performed for both
methods. First-order convergence is established for all variables in
their natural norms. A duality argument is employed to prove second
order superconvergence of the displacement at the cell
centers. Numerical results are presented in confirmation of the
theory.
\end{abstract}

\section{Introduction}
Stress-displacement mixed finite element (MFE) elasticity formulations
have been studied extensively due to their local momentum conservation
with continuous normal stress and locking-free approximation, see
\cite{brezzi2008mixed} and references therein. These methods result in
saddle point type algebraic systems, which may be expensive to
solve. In this work we develop two stress-displacement MFE methods for
elasticity on quadrilateral grids that can be reduced to symmetric and
positive definite cell centered systems using a mass lumping
technique. We have previously developed such methods on simplicial
grids in \cite{msmfe_simp}. Even though the formulation is similar,
the stability and error analysis on quadrilaterals differ
significantly from those on simplices. The methods are referred to as
multipoint stress mixed finite element (MSMFE) methods, adopting the
terminology of the multipoint stress approximation (MPSA) method
developed in
\cite{Jan-IJNME,nordbotten2015convergence,keilegavlen2017finite}.  Our
approach is motivated by the multipoint flux mixed finite element
(MFMFE) method
\cite{wheeler2006multipoint,Ing-Whe-Yot,WheXueYot-nonsym} for Darcy
flow, and its closely related multipoint flux approximation (MPFA)
method
\cite{aavat2002introduction,aavat1998discretization,edwards2002unstructured,edwards1998finite,klausen2006robust}. The
MFMFE method utilizes the lowest order Brezzi-Douglas-Marini $\BDM_1$
spaces on simplices and quadrilaterals \cite{brezzi1985two}, as well
as an enhanced Brezzi-Douglas-Duran-Fortin $\BDDF_1$ space
\cite{brezzi1991mixed} on hexahedral grids. There are two variants of
the MFMFE method - symmetric and non-symmetric. The symmetric version
is designed for simplices \cite{wheeler2006multipoint}, as well as
quadrilaterals and hexahedra that are $O(h^2)$-perturbations of
parallelograms and parallelepipeds
\cite{wheeler2006multipoint,Ing-Whe-Yot}. It is related to the
symmetric MPFA method
\cite{aavat2002introduction,Klausen.R;Winther.R2006,AEKWY}. The
symmetric MFMFE method is always well posed, but its convergence may
deteriorate on general quadrilaterals or hexahedra. The non-symmetric
MFMFE method \cite{WheXueYot-nonsym}, which is related to the
non-symmetric MPFA method
\cite{edwards1998finite,aavat2002introduction,klausen2006robust},
exhibits good convergence on general quadrilaterals or hexahedra, but
it may become ill-posed due to loss of coercivity if the grids are too
distorted.

The MSMFE methods on quadrilaterals we develop in this paper are
symmetric and are related to the symmetric MFMFE method.  The methods
are based on the $\BDM_1$ spaces on quadrilaterals. We consider the
formulation with weakly imposed stress symmetry, for which there exist
MFE spaces with $\BDM_1$ degrees of freedom for the stress. In this
formulation the symmetry of the stress is imposed weakly using a
Lagrange multiplier, which is a skew-symmetric matrix and has a
physical meaning of rotation.  Our first method, referred to as
MSMFE-0, is based on the spaces $\BDM_1\times\Qc_0\times\Qc_0$
developed in \cite{Awanou-rect-weak,arnold2015mixed}, using $\BDM_1$
stress and piecewise constant displacement and rotation. The $\BDM_1$
space has two normal degrees of freedom per edge, which can be
associated with the two vertices. An application of the trapezoidal
quadrature rule for the stress bilinear form results in localizing the
interaction of stress degrees of freedom around mesh vertices. The
stress is then locally eliminated and the method is reduced to a
symmetric and positive definite cell centered system for the
displacement and rotation. Our second method, MSMFE-1, is based on the
spaces $\BDM_1\times\Qc_0\times\Qc_1$, with continuous bilinear
rotation.  To the best of our knowledge, these spaces have not been
studied in the literature. In this method we employ the trapezoidal
quadrature rule both for the stress and the asymmetry bilinear
forms. This allows for further local elimination of the rotation after
the initial stress elimination, resulting in a symmetric and positive
definite cell centered system for the displacement only. To the best
of our knowledge, this is the first MFE method for elasticity on
quadrilaterals with such property.

We develop stability and error analysis for the two methods. The
stability arguments follow the framework established in
\cite{arnold2015mixed}, with modifications to account for the
quadrature rules. The argument in \cite{arnold2015mixed} explores
connections between stable mixed elasticity elements and stable mixed
Stokes and Darcy elements. In the case of the MSMFE-0 method, the two
stable pairs are $\mathcal{SS}_2\times\Qc_0$ for Stokes and
$\BDM_1\times\Qc_0$ for Darcy. Since the only term with quadrature is
the stress bilinear form, the stability argument in
\cite{arnold2015mixed} can be modified in a relatively straightforward
way.  Proving stability of the MSMFE-1 method is more difficult. In
this case the Stokes pair is $\mathcal{SS}_2\times\Qc_1$. The
difficulty comes from the fact that the quadrature rule is also
applied to the asymmetry bilinear forms, which necessitates
establishing an inf-sup condition for $\mathcal{SS}_2\times\Qc_1$ with
trapezoidal quadrature in the divergence bilinear form. We do this by
a macroelement argument motivated by \cite{stenberg1984analysis}.  It
is based on establishing a local macroelement inf-sup condition and
combining the locally constructed velocities to obtain the global
inf-sup condition. We note that the proof is very different from the
argument on simplices in \cite{msmfe_simp}. In particular, on
simplices one can establish a local inf-sup condition using vectors
that are zero on the boundary of the macroelement, which can be
utilized in the global construction. This is not the case on
quadrilaterals, which complicates the global construction
significantly. The reader is referred to
Section~\ref{inf-sup-stokes-section}, where the global Stokes inf-sup
condition is established under a smoothness assumption on the grid
given in \ref{M2}. We would like to note that this result is important
by itself, as it deals with the fundamental issue of inf-sup stability
for Stokes finite element approximation with quadrature.  We proceed
with establishing first order convergence for the stress in the
$H(\dvr)$-norm and for the displacement and rotation in the $L^2$-norm
for both methods on elements that are $O(h^2)$-perturbations of
parallelograms. This restriction is similar to the one in symmetric
MPFA and MFMFE methods
\cite{wheeler2006multipoint,Klausen.R;Winther.R2006}.  Again, the
arguments are very different from the simplicial case, since the map
to the reference element is non-affine (bilinear), which complicates
the estimation of the quadrature error. We further employ a duality
argument to prove second order superconvergence of the displacement at
the cell centers.

The rest of the paper is organized as follows. The model problem and
its MFE approximation are presented in Section~2.  The two methods and
their stability are developed in Sections 3 and 4, respectively. The
error analysis is performed in Section 5. Numerical results are
presented in Section 6.

\section{Model problem and its MFE approximation}

Let $\O$ be a simply connected bounded polygonal domain of $\R^2$ occupied by a
linearly elastic body. We write $\M$, $\S$ and $\N$ for the spaces of
$2\times 2$ matrices, symmetric matrices and skew-symmetric matrices,
all over the field of real numbers, respectively. The material
properties are described at each point $\x \in \O $ by a compliance
tensor $A = A(x)$, which is a symmetric, bounded and uniformly
positive definite linear operator acting from $\S$ to $\S$. We also
assume that an extension of $A$ to an operator $\M \to \M$ still
possesses the above properties. We will utilize the usual divergence
operator $\dvr$ for vector fields. When applied to a matrix field, it
produces a vector field by taking the divergence of each row. We will
also use the curl operator defined as
%
$    \curl{\phi} = (\partial_2 \phi, -\partial_1 \phi) $
%
for a scalar function $\phi$. Consequently, for a
vector field, the curl operator produces a matrix field, by acting row-wise.

Throughout the paper, $C$ denotes a generic positive constant that is
independent of the discretization parameter $h$. We will also use the
following standard notation. For a domain $G \subset \R^2$, the
$L^2(G)$ inner product and norm for scalar and vector valued functions
are denoted $\inp[\cdot]{\cdot}_G$ and $\|\cdot\|_G$,
respectively. The norms and seminorms of the Sobolev spaces
$W^{k,p}(G),\, k \in \R, p>0$ are denoted by $\| \cdot \|_{k,p,G}$ and
$| \cdot |_{k,p,G}$, respectively. The norms and seminorms of the
Hilbert spaces $H^k(G)$ are denoted by $\|\cdot\|_{k,G}$ and $| \cdot
|_{k,G}$, respectively. We omit $G$ in the subscript if $G = \O$. For
a section of the domain or element boundary $S$ we
write $\gnp[\cdot]{\cdot}_S$ and $\|\cdot\|_S$ for the $L^2(S)$ inner
product (or duality pairing) and norm, respectively. We will also use
the space $H(\dvrg; \O) = \{v \in L^2(\O, \R^2) : \dvr v \in L^2(\O)\}$
equipped with the norm
$\|v\|_{\dvr} = \left( \|v\|^2 + \|\dvr v\|^2 \right)^{1/2}$.

Given a vector field $f$ on $\Omega$ representing body forces, equations of static elasticity in Hellinger-Reissner form determine the stress $\sigma$ and the displacement $u$ satisfying the constitutive and equilibrium equations respectively:
\begin{align}
    A\s = \epsilon(u), \quad \dvr \s = f \quad \text{in } \O, \label{elast-1}
\end{align}
together with the boundary conditions 
\begin{align}
    u = g \ \text{ on } \Gd, \quad  \s\,n = 0 \ \text{ on } \Gn, \label{elast-2}
\end{align}
where $\epsilon(u) = \frac{1}{2}\left(\nabla u + (\nabla u)^T\right)$
and $\partial\O = \Gd \cup \Gn$. We assume for simplicity that $\Gd \neq \emptyset$.

We consider a weak formulation for \eqref{elast-1}--\eqref{elast-2}, in
which the stress symmetry is imposed weakly, using the
Lagrange multiplier $\g = \skew(\nabla u)$,
$\skew(\tau) = \frac12(\tau - \tau^T)$, from the space of
skew-symmetric matrices:
find $(\s, u, \g) \in \X \times V \times \W$ such that:
\begin{align}
    \inp[A\s]{\t} + \inp[u]{\dvr \t} + \inp[\g]{\t} &= \gnp[g]{\t\, n}_{\Gd}, &\forall \t &\in \X, \label{weak-1}\\
    \inp[\dvr \s]{v} &= \inp[f]{v}, &\forall v &\in V, \label{weak-2}\\
	\inp[\s]{\xi} &= 0, &\forall \xi &\in \W, \label{weak-3}
\end{align}
where the corresponding spaces are
$$\X = \left\{ \t\in H(\dvrg,\Omega,\M) : \t\,n = 0 \text{ on } \Gn  \right\}, \quad V = L^2(\Omega, \R^2), \quad W = L^2(\Omega, \N).
$$
Problem \eqref{weak-1}--\eqref{weak-3} has a unique solution \cite{arnold2007mixed}.

\subsection{Mixed finite element method} 
Let $\Tc_h$ be a shape-regular and quasi-uniform quadrilateral
partition of $\O$ \cite{ciarlet2002finite}, with $h=\max_{E\in \Tc_h}
\text{diam} (E)$. For any element $E \in \Tc_h$ there exists a bilinear
bijection mapping $F_E: \Eh \to E$, where $\Eh = [-1,1]^2$ is the reference
square. Denote the Jacobian matrix by $\DF_E$ and let $J_E =
|\operatorname{det} (\DF_E)|$. For $\x = F_E(\hat\x)$ we have 
$$
\DF^{-1}_E (\x) = (\DF_E)^{-1}(\hat \x), \qquad J_{F^{-1}_E}(\x) = \frac{1}{J_E(\hat\x)}.
$$
Let $\Eh$ has 
vertices $\rh_1 = (0,0)^T$, $\rh_2 = (1,0)^T$, $\rh_3 = (1,1)^T$ and 
$\rh_4 = (0,1)^T$ with unit outward normal vectors to the edges 
denoted by $\nh_i$, $i = 1,\ldots,4$, see Figure~\ref{elements}. 
We denote by $\r_i =
(x_i, y_i)^T$, $i = 1,\dots,4$, the corresponding vertices of the
element $E$, and by $n_i$, $i=1,\dots,4$, the corresponding 
unit outward normal vectors. The bilinear mapping $F_E$ and its Jacobian matrix
are given by
\begin{align}
& F_E(\rh) = \r_1 + \r_{21}\xh + \r_{41}\yh + (\r_{34} - \r_{21})\xh\yh, \label{mapping}\\
& \DF_E = \left[ \r_{21}, \r_{41}\right] 
+ \left[ (\r_{34}-\r_{21})\yh, (\r_{34}-\r_{21})\xh \right],  \label{df-map1}
\end{align}
where $\r_{ij} = \r_i - \r_j$. 
It is easy to see that the shape-regularity and quasi-uniformity of the grids imply that
$\forall E \in \Tc_h$,
\begin{align}
    \| \DF_E \|_{0,\infty, \Eh} \sim h, \quad \| \DF^{-1}_E \|_{0,\infty, \Eh} \sim h^{-1}, \quad \| J_E \|_{0,\infty, \Eh} \sim h^2 \quad \text{and} \quad \| J_{F_E^{-1}} \|_{0,\infty, \Eh} \sim h^{-2}, \label{scaling-of-mapping}
\end{align}
where the notation $a\sim b$ means that there exist positive constants
$c_0,\, c_1$ independent of $h$ such that $c_0 b \le a \le c_1 b$.

The finite element spaces 
$\X_h \times V_h \times \W_h^{k} \subset \X \times V \times \W$ are the triple
$(\BDM_1)^2 \times (\Qc_0)^2 \times (\Qc_k)^{2\times 2,skew}$, $k = 0,1$, where 
$\Qc_k$ denotes the space of polynomials of degree at most $k$ in each variable and
each row of an element of $\X_h$ is a vector in $\BDM_1$. On the reference square the
spaces are defined as
\begin{align}
\hat{\X}(\Eh) &= 
\left(\Pc_1(\Eh)^2 + r_1\,\curl(\xh^2 \yh) + s_1\, \curl (\xh\yh^2)\right)
\times \left(\Pc_1(\Eh)^2 + r_2\,\curl(\xh^2 \yh) + s_2\, \curl (\xh\yh^2)\right) \nonumber\\ 
&= \begin{pmatrix} \alpha_1 \xh + \beta_1 \yh + \g_1 + r_1\xh^2 + 2s_1\xh\yh & \alpha_2 \xh + \beta_2 \yh + \g_2 - 2r_1\xh\yh - s_1\yh^2 \\ \alpha_3 \xh + \beta_3 \yh + \g_3 + r_2\xh^2 + 2s_2\xh\yh & \alpha_4 \xh + \beta_4 \yh + \g_4 - 2r_2\xh\yh - s_2\yh^2 \end{pmatrix}, \label{ref-spaces}\\
    \Vh(\Eh) & = \left(\Qc_0(\Eh)\right)^2, \quad 
    \hat{\W}^k(\Eh) = \begin{pmatrix} 0 & p \\ -p & 0 \end{pmatrix}, \quad p\in \Qc_k(\Eh) \mbox{ for } k = 0,1,   \nonumber
\end{align}
where $\alpha_i, \beta_i, \g_i, r_i, s_i$ are real constants.  Note
that $\dvr \Xh(\Eh) = \Vh(\Eh)$ and for all $\th \in \Xh
(\Eh)$, $\th\, n_{\eh} \in \Pc_1(\eh)^2$ on any edge $\eh$ of $\Eh$.
It is well known \cite{brezzi1985two,brezzi1991mixed} that the degrees of freedom of
$\BDM_1(\Eh)$ can be chosen as the values of the normal components
at any two points on each edge $\eh \subset \partial \Eh$. In this
work we choose these points to be the vertices of $\eh$, see Figure
\ref{elements}. This is motivated by the trapezoidal quadrature rule,
introduced in the next section. The spaces on any element $E \in \Tc_h$ 
are defined via the transformations
\begin{align}\label{maps}
 \t \overset{\Pc}{\leftrightarrow} \hat{\t} : 
\t^T = \frac{1}{J_E} \DF_E \hat{\t}^T \circ F_E^{-1}, 
\quad
    v \leftrightarrow \hat{v} : v = \hat{v} \circ F_E^{-1}, \quad
    \xi \leftrightarrow \hat{\xi} : \xi = \hat{\xi} \circ F_E^{-1},
\end{align}
where $\t \in \X(E)$, $v \in V(E)$, and $\xi \in \W(E)$. Note that the Piola 
transformation (applied row-by-row) is used for $\X(E)$. It
preserves the normal components of the stress tensor
on edges, and it satisfies, for all $\t \in \X(E)$, $v \in V(E)$, and $\phi \in H^1(E)$,
\begin{equation}
(\dvr \t, v)_E = (\dvr \hat{\t}, \hat{v})_{\Eh}, 
\quad \langle \t\, n_e, v \rangle _e 
= \langle \hat{\t}\, \hat{n}_{\eh}, \hat{v} \rangle _{\eh},
\quad \text{and} \quad \curl \phi \overset{\Pc}{\leftrightarrow} \curl \hat\phi
. \label{prop-piola}
\end{equation}
The spaces on $\Tc_h$ are defined by
\begin{align} \label{spaces}
    \X_h &= \{\t \in \X: \t|_E \overset{\Pc}{\leftrightarrow} \hat{\t},\: \hat{\t} \in \hat{\X}(\Eh) \quad \forall E\in\mathcal{T}_h\}, \nonumber \\
    V_h &= \{v \in V: v|_E \leftrightarrow \hat{v},\: \hat{v} \in \hat{V}(\Eh) \quad \forall E\in\mathcal{T}_h\}, \\
    \W_h^0 & = \{\xi \in \W: \xi|_E \leftrightarrow \hat{\xi},
\: \hat{\xi} \in \hat{\W}^0(\Eh) \quad \forall E\in\mathcal{T}_h\}, \nonumber \\
\W_h^1 &  = \{\xi \in \mathcal{C}(\O,\N) \subset \W: \xi|_E \leftrightarrow \hat{\xi},\: \hat{\xi} \in \hat{\W}^1(\Eh) \quad \forall E\in\mathcal{T}_h\}. \nonumber
\end{align}
Note that $\W_h^1 \subset H^1(\Omega)$, since it contains continuous
piecewise $\Qc_1$ functions.

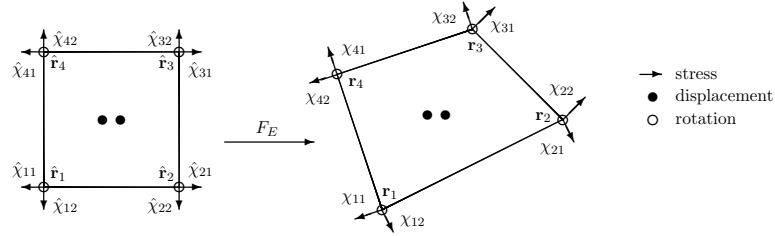
\begin{figure}	
	\setlength{\unitlength}{1.0mm}
\scalebox{.6}{
	\begin{picture}(70,45)(-70,-5)
	\thicklines
	\Line(5,5)(35,5)
	\Line(5,5)(5,35)
	\Line(5,35)(35,35)
	\Line(35,35)(35,5)    
	\Line(80,0)(120,20)
	\Line(80,0)(70,30)
	\Line(120,20)(100,40)
	\Line(100,40)(70,30)   
	
	\put(6.5,7){$\rh_{1}$}
	\put(30.5,7){$\rh_{2}$}
	\put(30.5,31.5){$\rh_{3}$}
	\put(6.5,31.5){$\rh_{4}$}
	\put(80,3){$\r_{1}$}
	\put(114,20){$\r_{2}$}
	\put(99,35.5){$\r_{3}$}
	\put(72.5,27.5){$\r_{4}$}
	
	\put(5,5){\vector(-1,0){5}} 
	\put(5,5){\vector(0,-1){5}} 
	\put(-2,8){$\ch_{11}$}
	\put(7,0){$\ch_{12}$}
	\put(35,35){\vector(1,0){5}} 
	\put(35,35){\vector(0,1){5}} 
	\put(28,0){$\ch_{22}$}
	\put(37,8){$\ch_{21}$}
	\put(5,35){\vector(-1,0){5}} 
	\put(5,35){\vector(0,1){5}} 
	\put(37,30){$\ch_{31}$}
	\put(28,37){$\ch_{32}$}
	\put(35,5){\vector(1,0){5}} 
	\put(35,5){\vector(0,-1){5}} 
	\put(-2,30){$\ch_{41}$}
	\put(7,37){$\ch_{42}$}
	\put(18,20){\circle*{2}}
	\put(22,20){\circle*{2}}
	\put(5,5){\circle{2}}
	\put(5,35){\circle{2}}
	\put(35,5){\circle{2}}
	\put(35,35){\circle{2}}
		
	\put(80,0){\vector(1,-2){2.5}} 
	\put(80,0){\vector(-3,-1){6}}
	\put(71,2){$\chi_{11}$}
	\put(84,-3){$\chi_{12}$}	
	\put(120,20){\vector(1,1){5}} 
	\put(120,20){\vector(1,-2){2.5}} 
	\put(117,26){$\chi_{22}$}
	\put(115,13){$\chi_{21}$}	
	\put(70,30){\vector(-3,-1){6}} 
	\put(70,30){\vector(-1,3){2}} 
	\put(71,35){$\chi_{41}$}
	\put(63,24){$\chi_{42}$}	
	\put(100,40){\vector(-1,3){2}} 
	\put(100,40){\vector(1,1){5}} 
	\put(104,40){$\chi_{31}$}
	\put(91,42){$\chi_{32}$}
	\put(90,21){\circle*{2}}
	\put(94,21){\circle*{2}}
	\put(80,0){\circle{2}}
	\put(120,20){\circle{2}}
	\put(70,30){\circle{2}}
	\put(100,40){\circle{2}}
		
	\put(45,15){\vector(1,0){20}} 
	\put(52,17){$F_E$}
	
	\put(137,30){\vector(1,0){5}}
	\put(145,29){\text{stress}}
	\put(140,25){\circle*{2}}
	\put(145,24){\text{displacement}}
	\put(140,20){\circle{2}}
	\put(145,19){\text{rotation}}
	\end{picture}
}
	\caption{Degrees of freedom of $\X_h\times V_h\times \W^1_h$.}
	\label{elements}	
\end{figure}

The MFE method for 
\eqref{weak-1}--\eqref{weak-3} is: find $(\s_h, u_h, \g_h) \in \X_h \times
V_h \times \W_h^k$ such that
\begin{align}
	(A \s_h,\t) + (u_h,\dvr{\t}) + (\g_h, \t) &= \gnp[g]{\t n}_{\Gd}, &&\t \in \X_h,  \label{h-weak-BDM-1}\\
	(\dvr \s_h,v) &= (f,v), && v \in V_h, \label{h-weak-BDM-2}\\
	(\s_h,\xi) &= 0, && \xi \in \W_h^k. \label{h-weak-BDM-3}
\end{align}
It is shown in \cite{arnold2015mixed} that the method
\eqref{h-weak-BDM-1}--\eqref{h-weak-BDM-3} in the case $k=0$ has a
unique solution and it is first order accurate for all variables in
their corresponding norms. The framework from \cite{arnold2015mixed}
can be used to analyze the case $k=1$. A drawback of the method is
that the resulting algebraic problem is a coupled
stress-displacement-rotation system of a saddle point type. In
this paper we develop two methods that utilize 
a quadrature rule and can be reduced to cell-centered systems
for displacement-rotation and displacement only, respectively.

\subsection{A quadrature rule}

Let $\varphi$ and $\psi$ be element-wise continuous functions on $\Omega$. We
denote by $(\varphi,\psi)_Q$ the application of the element-wise tensor
product trapezoidal quadrature rule for computing $(\varphi,\psi)$.
The integration on any element $E$ is performed by mapping to the
reference element $\Eh$. For $\tau, \, \chi \in \X_h$, we have
\begin{align*}
    \int_{E} A\tau : \chi \, d\x
= \int_{\Eh} \Ah\, \frac{1}{J_E} \th \DF_E^T : \frac{1}{J_E}\ch \DF_E^T\, J_E\, d\hat\x 
= \int_{\Eh} \Ah \th \frac{1}{J_E} \DF^T_E : \ch \DF^T_E\, d\hat\x.
\end{align*}
The quadrature rule on an element $E$ is then defined as 
\begin{equation}
    (A\tau,\chi)_{Q,E} \equiv 
(\Ah \th \frac{1}{J_E} \DF^T_E, \ch \DF^T_E)_{\hat Q,\hat E}
\equiv \frac{|\Eh|}{4} \sum_{i=1}^{4} 
\Ah(\rh_i) \th(\rh_i) \frac{1}{J_E(\rh_i)} \DF^T_E(\rh_i) 
: \ch(\rh_i) \DF^T_E(\rh_i). \label{quad-rule-def}
\end{equation}
The global quadrature rule is defined as $(A\tau,\chi)_Q \equiv
\sum_{E \in \Tc_h} (A\tau, \chi)_{Q,E}$.  We note that the quadrature rule can be 
defined directly on a physical element $E$:
\begin{align}
(A\t,\chi)_{Q,E} =\frac{1}{2}\sum_{i=1}^4|T_i|A(\r_i)\t(\r_i):\chi(\r_i), 
\label{quad-rule-physical}
\end{align}
where $|T_i|$ is the area of triangle formed by the two edges sharing $\r_i$.

Recall that the stress
degrees of freedom are the two normal components per edge evaluated at
the vertices, see Figure~\ref{elements}.  For an element vertex
$\r_i$, the tensor $\chi(\r_i)$ is uniquely determined by its normal
components to the two edges that share that vertex. Since the basis
functions associated with a vertex are zero at all other vertices, the
quadrature rule \eqref{quad-rule-def} decouples the degrees of freedom
associated with a vertex from the rest of the degrees of freedom,
which allows for local stress elimination.

We also employ the trapezoidal quadrature rule for the stress-rotation
bilinear forms in the case of bilinear
rotations. For $\t \in \X_h,\, \xi \in \W^1_h$, we have
\begin{equation}
(\t, \xi)_{Q,E} \equiv 
\left(\frac{1}{J_E}\th\DF_E^T,\hat{\xi}J_E\right)_{\hat{Q},\hat{E}} 
\equiv \frac{|\hat{E}|}{4}\sum_{i=1}^{4}\th(\rh_i)\DF_E(\rh_i)^T
:\hat{\xi}(\rh_i). \label{def}
\end{equation}

The next lemma shows that the quadrature rule \eqref{quad-rule-def} produces a
coercive bilinear form.
\begin{lemma}\label{coercivity-lemma}
The bilinear form $\inp[A\tau]{\chi}_Q$ is an inner product on $\X_h$
and $\inp[A\tau]{\tau}_Q^{1/2}$ is a norm in $\X_h$ equivalent to
$\|\cdot \|$, i.e., there exist constants $0 < \alpha_0 \le \alpha_1$
independent of $h$ such that
\begin{equation}\label{coercivity}
\alpha_0 \|\tau\|^2 \le \inp[A\tau]{\tau}_Q \le \alpha_1 \|\tau\|^2 \quad 
\forall \tau\in\X_h.
\end{equation}
Furthermore, $(\xi,\xi)^{1/2}_Q$ is a norm in $\W^1_h$ equivalent to
$\|\cdot\|$, and $\forall \, \tau \in \X_h$, $\xi \in \W^1_h$,
$(\tau,\xi)_Q \le C \|\tau\|\|\xi\|$.
\end{lemma}
\begin{proof}
The proof follows the argument of Lemma~2.2 in \cite{msmfe_simp}, 
using \eqref{quad-rule-physical}.
\end{proof}

\section{The multipoint stress mixed finite element method with constant 
rotations (MSMFE-0)}

Let $\Pc_0$ be the $L^2$-orthogonal
projection onto $\X^{\RT}_h\,n$, the space of piecewise constant vector-valued 
functions on the trace of $\Tc_h$ on $\partial \O$:
\begin{align}
\forall \phi \in L^2(\partial\O), \quad \langle \phi-\Pc_0\phi, \t \,n\rangle_{\partial \O} =0,\quad \forall \t\in \X^{\RT}_h. \label{rhs-proj}
\end{align}
Our first method, referred to as MSMFE-0, is: 
find $\sigma_h \in \X_h,\, u_h \in V_h$, and $\g_h \in \W_h^0$ such that
\begin{align}
(A \s_h,\t)_Q + (u_h,\dvr{\t}) + (\g_h, \t) &= \gnp[\Pc_0 g]{\t\,n}_{\Gd}, 
& \t &\in \X_h, \label{h-weak-P0-1}\\
(\dvr \sigma_h,v) &= (f,v), & v &\in V_h, \label{h-weak-P0-2}\\
(\sigma_h,\xi) &= 0, &\xi &\in \W_h^0. \label{h-weak-P0-3}
\end{align}
The Dirichlet data is incorporated into the scheme as $\Pc_0 g$, which
is necessary for the optimal approximation of the boundary condition
term. 

\begin{theorem}
The method \eqref{h-weak-P0-1}--\eqref{h-weak-P0-3} has a unique solution.
\end{theorem}
\begin{proof}
Using classical stability theory of mixed finite element methods,
the required Babu\v{s}ka-Brezzi stability conditions \cite{brezzi1991mixed} are:
\begin{enumerate}[label={\bf(S\arabic*)}]
	\item \label{S1-P0} There exists $c_1 > 0$ such that 
	\begin{align}
	c_1\| \t \|_{\dvrg} \le \inp[A\t]{\t}^{1/2}_{Q}
	\end{align} 
	for $\t \in \X_h$ satisfying $\inp[\dvr \t]{v} = 0$ and $\inp[\t]{\xi} = 0$ for all $(v,\xi) \in V_h\times \W_h^0$.
	\item \label{S2-P0} There exists $c_2 > 0$ such that 
	\begin{align}
	\inf_{0\neq(v,\xi)\in V_h\times \W_h^0 } \sup_{0\neq \t\in\X_h} \frac{\inp[\dvr\t]{v}+\inp[\t]{\xi}}{\| \t \|_{\dvrg} \left( \|v\| + \|\xi\| \right)}  \ge c_2.  \label{inf-sup-P0}
	\end{align}
\end{enumerate}
Using \eqref{prop-piola} and $\dvrg \Xh(\Eh) = \Vh(\Eh)$, 
the condition $(\dvrg \tau,v)=0,\, \forall v \in V_h$ implies that 
$\dvr \tau =0$. Then \ref{S1-P0} follows from \eqref{coercivity}.
The inf-sup condition \ref{S2-P0} has been shown in \cite{arnold2015mixed}.
\end{proof}

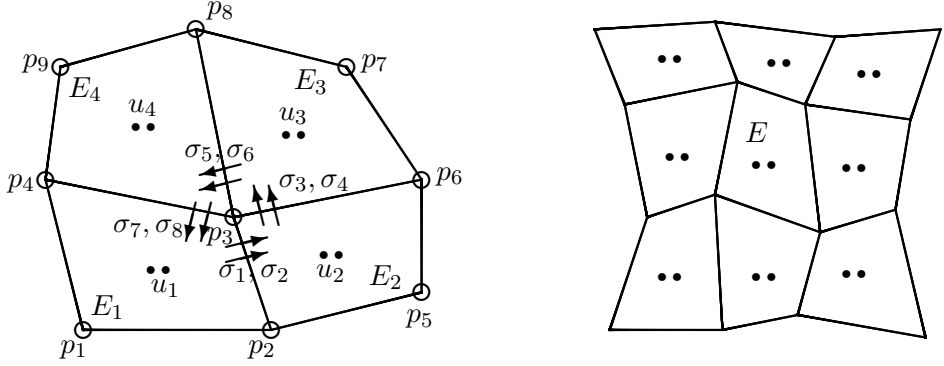
\begin{figure}		
	\setlength{\unitlength}{1.0mm}
	\begin{picture}(100,45)(-20,0)
	\thicklines
	\Line(10,5)(5,25)
	\Line(10,5)(35,5)
	\Line(5,25)(30,20)
	\Line(30,20)(35,5)  
	
	\Line(30,20)(55,25)
	\Line(55,25)(55,10)
	\Line(55,10)(35,5)  
	
	\Line(30,20)(25,45)
	\Line(25,45)(45,40)
	\Line(45,40)(55,25)
	
	\Line(5,25)(7,40)
	\Line(7,40)(25,45)
	
	\put(19,13){\circle*{1}}
	\put(21,13){\circle*{1}}
	
	\put(17,32){\circle*{1}}
	\put(19,32){\circle*{1}}
	
	\put(37,31){\circle*{1}}
	\put(39,31){\circle*{1}}
	
	\put(42,15){\circle*{1}}
	\put(44,15){\circle*{1}}
	
	\put(10,5){\circle{2}}
	\put(35,5){\circle{2}}
	\put(30,20){\circle{2}}
	\put(5,25){\circle{2}}    
	\put(7,40){\circle{2}}
	\put(25,45){\circle{2}}
	\put(45,40){\circle{2}}
	\put(55,10){\circle{2}}	
	\put(55,25){\circle{2}}
	
	\put(31,25){\vector(-4,-1){5.5}} 
	\put(31,27){\vector(-4,-1){5.5}} 
	\put(25,22){\vector(-1,-4){1.3}}
	\put(27,22){\vector(-1,-4){1.3}} 
	\put(29,16){\vector(4,1){5.5}}
	\put(29,14){\vector(4,1){5.5}} 
	\put(34,19){\vector(-1,4){1.3}}
	\put(36,19){\vector(-1,4){1.3}} 
	
	\put(11,7){$E_1$}
	\put(48,11){$E_2$}
	\put(8,36){$E_4$}
	\put(38,37){$E_3$}
	\put(19,10){$u_1$}
	\put(16,34){$u_4$}
	\put(36,33){$u_3$}
	\put(41,12){$u_2$}
	\put(7,2){$p_1$}
	\put(32,2){$p_2$}		
	\put(53,6){$p_5$}	
	\put(0,24){$p_4$}
	\put(26.5,17){$p_3$}		
	\put(57,25){$p_6$}
	\put(2,40){$p_9$}
	\put(26.5,47){$p_8$}		
	\put(47,40){$p_7$}		
	\put(28,12){$\s_1,\s_2$}
	\put(36,24){$\s_3,\s_4$}
	\put(23.5,28){$\s_5,\s_6$}
	\put(14,18){$\s_7,\s_8$}

	\Line(80,5)(85,20)
	\Line(85,20)(82,35)
	\Line(82,35)(78,45)
	
	\Line(80,5)(95,5)
	\Line(85,20)(94,23)
	\Line(82,35)(97,38)
	\Line(78,45)(94,46)
	
	\Line(95,5)(94,23)
	\Line(94,23)(97,38)
	\Line(97,38)(94,46)
	
	\Line(105,7)(108,18)
	\Line(108,18)(106,35)
	\Line(106,35)(110,44)
	
	\Line(95,5)(105,7)
	\Line(94,23)(108,18)
	\Line(97,38)(106,35)
	\Line(94,46)(110,44)
	
	\Line(122,4)(118,21)
	\Line(118,21)(120,33)
	\Line(120,33)(124,45)
	
	\Line(105,7)(122,4)
	\Line(108,18)(118,21)
	\Line(106,35)(120,33)
	\Line(110,44)(124,45)
	\put(87,12){\circle*{1}}
	\put(89,12){\circle*{1}}
	\put(99.5,12){\circle*{1}}
	\put(101.5,12){\circle*{1}}
	\put(111.5,12.5){\circle*{1}}
	\put(113.5,12.5){\circle*{1}}
	
	\put(88,28){\circle*{1}}
	\put(90,28){\circle*{1}}
	\put(99.5,27){\circle*{1}}
	\put(101.5,27){\circle*{1}}
	\put(111.5,26.5){\circle*{1}}
	\put(113.5,26.5){\circle*{1}}
	\put(98,30){$E$}
	
	\put(87,41){\circle*{1}}
	\put(89,41){\circle*{1}}
	\put(101.5,40.5){\circle*{1}}
	\put(103.5,40.5){\circle*{1}}
	\put(113.5,39){\circle*{1}}
	\put(115.5,39){\circle*{1}}
	
	\end{picture}
	\caption{Finite elements sharing a vertex (left) and displacement stencil (right)}
	\label{stencil}
\end{figure}

\subsection{Reduction to a cell-centered displacement-rotation system}

The algebraic system that arises from 
\eqref{h-weak-P0-1}--\eqref{h-weak-P0-3} is of the form
\begin{equation}\label{sp-matrix}
    \begin{pmatrix} \As & \Au^T & \Ag^T \\ - \Au & 0   & 0 \\ - \Ag & 0   & 0 \end{pmatrix} 
    \begin{pmatrix} \sigma \\ u \\ \g \end{pmatrix} = 
    \begin{pmatrix} g \\ - f \\ 0 \end{pmatrix},
\end{equation}
where $(\As)_{ij} = (A\tau_j,\tau_i)_Q$, $(\Au)_{ij} = (\dvr\tau_j,
v_i)$, and $(\Ag)_{ij} = (\tau_j,\xi_i)$.  The method can be reduced
to solving a cell-centered displacement-rotation system as
follows. Since the quadrature rule $(A\sigma_h,\tau)_Q$ localizes the
basis functions interaction around mesh vertices, the matrix $\As$ is
block-diagonal with $2k\times 2k$ blocks associated with vertices, where $k$
is the number of elements that share the vertex, see
Figure~\ref{stencil} (left) for an example with $k=4$.
Lemma~\ref{coercivity-lemma} implies that the blocks are symmetric and 
positive definite. Therefore the
stress $\sigma_h$ can be easily eliminated by solving small local
systems, resulting in the cell-centered displacement-rotation system 
\begin{equation}\label{msmfe0-system}
    \begin{pmatrix} \Au\As^{-1}\Au^T & \Au\As^{-1}\Ag^T \\ 
\Ag\As^{-1}\Au^T & \Ag\As^{-1}\Ag^T \end{pmatrix} 
    \begin{pmatrix} u \\ \g \end{pmatrix} =
    \begin{pmatrix} \tilde{f} \\ \tilde{h} \end{pmatrix}.
\end{equation}
The displacement and rotation stencils for an element $E$ include all 
elements that share a vertex with $E$, see Figure~\ref{stencil} (right)
for an example of the displacement stencil. The matrix in 
\eqref{msmfe0-system} is clearly symmetric. Furthermore, for any
$\begin{pmatrix} v^T & \xi^T \end{pmatrix} \neq 0$, 
\begin{equation}\label{spd-matrix}
\begin{pmatrix} v^T & \xi^T \end{pmatrix}
\begin{pmatrix} \Au\As^{-1}\Au^T & \Au\As^{-1}\Ag^T 
\\ \Ag\As^{-1}\Au^T & \Ag\As^{-1}\Ag^T \end{pmatrix}
\begin{pmatrix} v \\ \xi \end{pmatrix} 
= (\Au^Tv + \Ag^T\xi)^T\As^{-1}(\Au^Tv + \Ag^T\xi) > 0,
\end{equation}
due to the inf-sup condition \ref{S2-P0}, which implies that the matrix
is positive definite.

\begin{remark}
The MSMFE-0 method is more efficient than the original MFE method,
since it involves a smaller system, which is symmetric and positive
definite. We note that further reduction in the system is not
possible. In the next section we develop a method with continuous
bilinear rotations and a trapezoidal quadrature rule applied to the
stress-rotation bilinear forms. This allows for further local
elimination of the rotation, resulting in a cell-centered system for
the displacement only.
\end{remark}

\section{The multipoint stress mixed finite element method with bilinear 
rotations (MSMFE-1)}

In the second method, referred to as MSMFE-1, we take $k=1$ in \eqref{ref-spaces}
and apply the quadrature rule to both the stress
bilinear form and the stress-rotation bilinear forms. The method is: find
$\sigma_h \in \X_h,\, u_h \in V_h$ and $\g_h \in \W_h^1$ such that
\begin{align}
(A \s_h,\t)_Q + (u_h,\dvr{\t}) + (\g_h,\t)_Q &= \gnp[\Pc_0 g]{\t \, n}_{\Gd}, &\t &\in \X_h, \label{h-weak-P1-1}\\
(\dvr \sigma_h,v) &= (f,v), &v &\in V_h,\label{h-weak-P1-2}\\
(\sigma_h,\xi)_Q &= 0, &\xi &\in \W_h^1. \label{h-weak-P1-3}
\end{align}
The stability conditions for the MSMFE-1 method are as follows:
\begin{enumerate}[label={\bf(S\arabic*)}]
	\setcounter{enumi}{2}
	\item \label{S3-P1}There exists $c_3 > 0$ such that 
	$$ c_3\| \t \|_{\dvrg}^2 \le \inp[A\t]{\t}_{Q}, $$
	for $\t \in \X_h$ satisfying $\inp[\dvr \t]{v} = 0$ and $\inp[\t]{\xi}_{Q} = 0$ for all $(v,\xi) \in V_h\times \W_h^1$.
	\item \label{S4-P1}There exists $c_4 > 0$ such that 
	\begin{align}
	\inf_{0\neq(v,\xi)\in V_h\times \W_h^1 } \sup_{0\neq \t\in\X_h} \frac{\inp[\dvr\t]{v}+\inp[\t]{\xi}_Q}{\| \t \|_{\dvrg} \left( \|v\| + \|\xi\| \right)}  \ge c_4. \label{inf-sup-P1}
	\end{align} 
\end{enumerate}

\subsection{Well-posedness of the MSMFE-1 method}
The stability condition \ref{S3-P1} holds, since the spaces $\X_h$ and
$V_h$ are as in the MSMFE-0 method. However, 
\ref{S4-P1} is different, due to the quadrature rule in $\inp[\tau]{\xi}_Q$, 
and it needs to be verified. The next theorem, proved in \cite{msmfe_simp},
provides sufficient conditions for a triple $\X_h\times V_h \times
\W^1_h$ to satisfy \ref{S4-P1}, where we adopt the notation 
$b(q,w)= (\dvr q,w)$ and $b(q,w)_Q = (\dvr q,w)_Q$.
\begin{theorem}\label{inf-sup-P1-proof}
Suppose that $S_h \subset H(\dvr;\Omega)$ and $U_h \subset L^2(\Omega)$ satisfy
\begin{align}
\inf\limits_{0\neq r\in U_h} \sup\limits_{0\neq z\in S_h} \frac{b(z,r)}{\| z \|_{\dvr} \|r\| }  \ge c_5, \label{darcy-pair}
\end{align}
that $Q_h \subset H^1(\Omega,\R^2)$ and $W_h \subset L^2(\Omega)$ are such 
that $(w,w)_Q^{1/2}$ is a norm in $W_h$ equivalent to $\|w\|$ and
\begin{align}
\inf\limits_{0\neq w\in W^1_h} \sup\limits_{0\neq q\in Q_h} 
\frac{b(q,w)_Q}{\| q \|_{1} \|w\| }  \ge c_6. \label{stokes-pair} 
\end{align}
and that
	\begin{align}
	\curl Q_h \subset S_h \times S_h. \label{curl-condition}
	\end{align}
Then, $\X_h = S_h \times S_h \subset H(\dvrg;\Omega,\M)$, $V_h = U_h \times U_h \subset L^2(\Omega, \mathbb{R}^2)$ and 
$\W^1_h =\left\{\xi: \xi =\begin{pmatrix}
	0 & w \\
	-w & 0
	\end{pmatrix},\, w \in W_h \right \} \subset L^2(\Omega,\N)$ satisfy \ref{S4-P1}.
\end{theorem}
\begin{remark}
Condition \eqref{darcy-pair} states that $S_h\times U_h$ is a stable Darcy pair.
Condition \eqref{stokes-pair} states that $Q_h\times W_h$ is a stable 
Stokes pair with quadrature.
\end{remark}

\begin{lemma}\label{aux-lemma}
Conditions \eqref{darcy-pair} and \eqref{curl-condition} hold for 
$\X_h\times V_h \times \W^1_h$ defined  in \eqref{ref-spaces} and 
\eqref{spaces}.
\end{lemma}
\begin{proof}
According to the definition \eqref{ref-spaces}, we take
\begin{equation*}
S_h = \{ z\in H(\dvrg;\Omega): z|_E \overset{\Pc}\leftrightarrow \hat z \in \BDM_1(\Eh), \,
z\cdot n = 0 \mbox{ on } \Gamma_N\}, 
\end{equation*}
\begin{equation*}
U_h = \{r \in L^2(\Omega): r|_E \leftrightarrow \hat r \in \Qc_0(\Eh)\}, \quad
W_h = \{w \in H^1(\Omega): w|_E \leftrightarrow \hat w \in \Qc_1(\Eh)\}.
\end{equation*}
We note that
$W_h$ satisfies the norm equivalence $(w,w)_Q^{1/2} \sim \|w\|$,
see Lemma~\ref{coercivity-lemma}.
The boundary condition in $S_h$ is needed to guarantee the essential
boundary condition in $\X_h$ on $\Gamma_N$. Since $\BDM_1 \times \Qc_0$ is a stable
Darcy pair \cite{brezzi1991mixed}, \eqref{darcy-pair} holds.
Next, following the construction in \cite{arnold2015mixed},
we take $\mathcal{SS}_2(\Eh)$ to be the reduced bi-quadratics (serendipity) space 
\cite{ciarlet2002finite},
\begin{align*}
\mathcal{SS}_2(\Eh) = \Pc_2(\Eh) +\mbox{span}\{\xh^2\yh,\xh\yh^2\},
\end{align*}
and define the space $Q_h$ as
\begin{align}
Q_h =\{q\in H^1(\O,\R^2): q_i|_{E} \leftrightarrow \hat{q}_i \in 
\mathcal{SS}_2(\hat{E}), \,i=1,2,\,\forall E\in \mathcal{T}_h, \, 
q = 0 \mbox{ on } \Gamma_N\}. \label{stokes-vel-space}
\end{align}
One can verify that $\curl \mathcal{SS}_2(\Eh) \subset \mathcal{BDM}_1(\Eh)$.
Due to \eqref{prop-piola}, $\curl Q_h \subset S_h\times S_h$, not considering 
the boundary condition on $\Gamma_N$. Finally, since for $q \in Q_h$ we have
$q = 0$ on $\Gamma_N$, then $(\curl q) \, n = 0$ on $\Gamma_N$, see 
\cite[Lemma 4.2]{msmfe_simp}. 
\end{proof}
To prove \ref{S4-P1}, it remains to show that \eqref{stokes-pair} holds.
It is shown in \cite{stenberg1984analysis} that $\mathcal{SS}_2-\Qc_1$ is 
a stable Stokes pair on rectangular grids. We need to verify that it is a stable 
Stokes pair with quadrature on quadrilaterals, which we do next.

\subsubsection{The inf-sup condition for the Stokes problem}\label{inf-sup-stokes-section}

We prove \eqref{stokes-pair} using a modification of the macroelement
technique presented in \cite{stenberg1984analysis}. The idea is to
establish first a local inf-sup condition and then combine locally
constructed velocity vectors to prove the global inf-sup condition.
We recall that in \cite{stenberg1984analysis}, it was sufficient to
consider velocity functions that vanish on the macroelement boundary
in order to control the pressures locally. However, due to the
quadrature rule, this is not true in our case. We show how a similar
result can be obtained without restricting the velocity basis
functions on the macroelement boundary, under a smoothness assumption
on the grid $\Tc_h$.

We
consider the span of all edge degrees of freedom of $Q_h(E)$ and
denote it by $Q^e_h(E)$. Let
\begin{align*}
N_E = \{w \in W_h(E) :\, b(q,w)_{Q,E} = 0,\, \forall q\in Q^e_h(E) \}.
\end{align*}
We make the following assumptions on the mesh.
 \begin{enumerate} [label={\bf(M\arabic*)},align=left]
    \item \label{M1} Each element $E$ has at most one edge on $\Gamma_N$.

\item \label{M2} The mesh size $h$ is sufficiently small and there
  exists a constant $C$ such that for every pair of neighboring
  elements $E$ and $\tilde E$ such that $E$ or $\tilde E$ is a
  non-parallelogram, and every pair of edges $e \subset \partial E \setminus \partial \tilde E$,
  $\tilde e \subset \partial \tilde E \setminus \partial E$ that share a vertex,
	\begin{align*}
	\|\r_{e}-\r_{\tilde e}\|_{\mathbb{R}^2} \leq Ch^2,
	\end{align*}
	where $\r_e$ and $\r_{\tilde e}$ are the vectors corresponding to $e$ 
and $\tilde e$, respectively, and $\|\cdot\|_{\mathbb{R}^2}$ is the Euclidean vector norm.
\end{enumerate}

\begin{remark}
Condition \ref{M1} is needed to establish a local inf-sup condition.
Condition \ref{M2} is needed to combine the local results and prove
the global inf-sup condition \eqref{stokes-pair}. We note that it is
required only for non-parallelogram elements. It is a mesh smoothness
condition. For example, it is satisfied if the mesh is generated by a
$C^2$ map of a uniform reference grid. The condition on the mesh size
$h$ is given in the proof of Lemma~\ref{macro-lemma-2}.
\end{remark}

\input fig-M.tex

\begin{lemma}\label{N_M-lemma}
Let \ref{M1} hold. If $E$ is
a parallelogram, then $N_E$ is one-dimensional, consisting of
functions that are constant on $E$; otherwise $N_E = 0$.
\end{lemma}

\begin{proof}
For any $q\in Q_h(E),\, w\in W_h(E)$, we have 
\begin{align*}
 b(q, w)_{Q,E} = (\tr(\nabla q),w)_{Q,E} = 
\frac{1}{4}\sum_{j=1}^4\tr\left[ \DF^{-T}_{E}(\hat{\r}_j) 
\hat{\nabla}\hat{q}(\hat{\r}_j)\right]\hat{w}(\hat{\r}_j)J_{E}(\hat{\r}_j).
 \end{align*}
Consider first an element with no edges on $\Gamma_N$.  Denote the
basis functions for $Q^e_h(E)$ by $q_i = q^n_i + q^t_i, \,
i=1,\dots,4$, see Figure~\ref{macroelements} (left). Without loss of
generality, assume that the edge corresponding to $q_1$ is
horizontal, i.e., $y_2 - y_1 = 0$, $x_2-x_1 \neq 0$, $x_3-x_4\neq 0$,
$y_4-y_1 \neq 0$, and $y_3-y_2\neq 0$. A direct calculation gives
\begin{align}
 b(q^t_1, w)_{Q,E}& =(y_4-y_1)w(\r_1) +(y_2-y_3)w(\r_2), \label{q-tan-1}\\
 b(q^n_1, w)_{Q,E}& =(y_1-y_2)w(\r_1) +(y_2-y_1)w(\r_2), \label{q-norm-1}\\ 
 b(q^t_2, w)_{Q,E}& =(x_2-x_1)w(\r_2) +(x_4-x_3)w(\r_3), \label{q-tan-2} \\
 b(q^n_2, w)_{Q,E}& =(x_2-x_3)w(\r_2) +(x_3-x_2)w(\r_3), \label{q-norm-2}\\
 b(q^t_3, w)_{Q,E}& =(y_2-y_3)w(\r_3) +(y_4-y_1)w(\r_4), \label{q-tan-3}\\
 b(q^n_3, w)_{Q,E}& =(y_3-y_4)w(\r_3) +(y_4-y_3)w(\r_4). \label{q-norm-3}\\
 b(q^t_4, w)_{Q,E}& =(x_2-x_1)w(\r_1) +(x_4-x_3)w(\r_4), \label{q-tan-4}\\
 b(q^n_4, w)_{Q,E}& =(x_1-x_4)w(\r_1) +(x_4-x_1)w(\r_4). \label{q-norm-4}
 \end{align}
Let us set the above quantities equal to zero. Consider the vertically oriented
 edges of $E$. From \eqref{q-tan-2} and \eqref{q-tan-4} 
we get
\begin{align}\label{element-1-tang}
 w(\r_2) = w(\r_3) \frac{x_4-x_3}{x_1-x_2}, \quad 
 w(\r_1) = w(\r_4) \frac{x_4-x_3}{x_1-x_2}.
 \end{align}
If $x_2\neq x_3$, we also get from \eqref{q-norm-2} that $w(\r_2) =
w(\r_3)$. This together with \eqref{element-1-tang} implies that
$w(\r_1) = w(\r_4)$. Similarly, if $x_1\neq x_4$, it follows from
\eqref{q-norm-4} that $w(\r_1)=w(\r_4)$, and \eqref{element-1-tang}
implies that $w(\r_2) = w(\r_3)$. Finally, if $x_2=x_3$ and $x_1=x_4$,
we arrive to the same conclusion directly from
\eqref{element-1-tang}.

Next, consider the edges corresponding to $q_1$ and $q_3$. From 
\eqref{q-tan-3} we get
\begin{align}
  w(\r_3) &= w(\r_4) \frac{y_1-y_4}{y_2-y_3} \label{element-1-tang-3}.
  \end{align} 	
If $y_3\neq y_4$, \eqref{q-norm-3} implies that $w(\r_3) = w(\r_4)$.
If $y_3=y_4$, since $y_1=y_2$, we obtain
from \eqref{element-1-tang-3} that $w(\r_3) = w(\r_4)$. Hence, $w$
must be constant on $E$.

We next consider the case when one of the edges of $E$ is on
$\Gamma_N$. Let this be the edge associated with $q_1$,
 as shown on Figure \ref{macroelements} (middle). Since the
above argument above did not use \eqref{q-tan-1} or \eqref{q-norm-1}, the
conclusion still applies.

Finally, if $w$ is a non-zero constant in $N_E$, setting the equations
\eqref{q-norm-1}--\eqref{q-norm-4} to zero implies that $E$ is a 
parallelogram.
\end{proof}

\begin{theorem}\label{macro-inf-sup}
If \ref{M1}--\ref{M2} are satisfied, then \eqref{stokes-pair} holds.
\end{theorem}
The proof of Theorem \ref{macro-inf-sup} is based on several auxiliary lemmas.

\begin{lemma}\label{macro-lemma-1}
If \ref{M1} holds, then there exists a constant $\beta > 0$ independent of $h$ 
such that, 
\begin{align*}
\forall \, T \in \Tc_h, \quad
\sup_{0\neq q \in Q^e_h(E)}\frac{b(q,w)_{Q,E} }{\|q\|_{1,E}} 
\geq \beta\|w\|_{E},\, \forall w \in W_h(E)/N_E.
\end{align*}
\end{lemma}
\begin{proof}
The proof follows from Lemma~\ref{N_M-lemma} 
and a scaling argument, see
\cite[Lemma 3.1]{stenberg1984analysis}.
\end{proof}

For $E \in \Tc_h$,
let $\mathbb{P}_h^E$ denote the $L^2$-projection from $W_h(E)$ onto $N_E$.

\begin{lemma}\label{macro-lemma-2}
If \ref{M1} and \ref{M2} hold, then there exists a constant $C_1 >0$, such that 
for every $w \in W_h$ and for every $E \in \Tc_h$ that is either a 
non-parallelogram or a parallelogram that neighbors parallelograms,
there exists $q_E \in Q^e_h(E)$ satisfying
\begin{align}\label{macro-lemma-2-ineq}
b(q_E,w)_{Q}  \geq C_1 \|(I-\mathbb{P}_h^E)w\|_E^2 \quad \mbox{and} 
\quad \|q_E\|_1 \leq \|(I-\mathbb{P}_h^E)w\|_E.
\end{align}
\end{lemma}
\begin{proof}
Let $w \in W_h$. Due to
Lemma~\ref{N_M-lemma}, if $E$ is not a parallelogram, then
$\mathbb{P}_h^E w = 0$ on $E$. Otherwise, $\mathbb{P}_h^E w$ is the
mean value of $w$ on $E$. Lemma \ref{macro-lemma-1} implies that for every 
$E$ there exists $q_E \in Q^e_h(E)$ such that
\begin{align}
b(q_E,w)_{Q,E} = b(q_E,(I-\mathbb{P}_h^E)w)_{Q,E} 
\geq C \|(I-\mathbb{P}_h^E)w\|^2_{E} \quad \mbox{and} 
\quad \|q_E\|_{1,E} \leq \|(I-\mathbb{P}_h^E)w\|_{E}. \label{local-inf-sup-const}
\end{align}
We note that $q_E$ does not vanish outside of $E$; however, we will show 
that under assumption \ref{M2}
\begin{align}
b(q_E,w)_{Q, \O\setminus E} \geq 0. \label{q_m_positive}
\end{align} 
In order to prove \eqref{q_m_positive} let us consider a neighboring
element $\tilde E$, see Figure~\ref{macroelements} (right). Let $q_E =
\sum_{i=1}^4 \alpha_i q_i$. We first consider a non-parallelogram $E$.
Consider the
tangential degree of freedom $q^t_1$, associated with the edge
shared by $E$ and $\tilde E$. Using \eqref{q-tan-1}, we have
\begin{align}\label{q-tan-1-a}
b(q^t_1, w)_{Q,E} = (y_4-y_1)w(\r_1) + (y_2-y_3)w(\r_2) 
:= \sum_{j=1}^4 \delta^t_{1,j}w(\r_j),
\end{align}
where $ \delta^t_{1,1} = (y_4-y_1),\,\delta^t_{1,2} = (y_2-y_3) $ and 
$\delta^t_{1,j} =0 $ for $j=3,4$. For $q^n_1$, using 
\eqref{q-norm-1}, we have
\begin{align}\label{q-norm-1-a}
b(q^n_1, w)_{Q,E} = (y_1-y_2)w(\r_1) + (y_2-y_1)w(\r_2) 
:= \sum_{j=1}^4 \delta^n_{1,j}w(\r_j).
\end{align}
Using a similar expression for the rest of the degrees of freedom, we obtain
\begin{align*}
b(q_E,w)_{Q,E} = \sum_{i=1}^4 \alpha_i b(q_i,w)_{Q,E}
= \sum_{i=1}^4\sum_{j=1}^4
\alpha_i\delta_{i,j} w(\r_j),
\end{align*}
where $\delta_{i,j} = \delta^n_{i,j} + \delta^t_{i,j}$.
We note that for all $i,j$, $\delta_{i,j} = 0$ or $|\delta_{i,j}| =
O(h)$. Using \eqref{q-tan-3}, we also compute
\begin{align}\label{q-tan-3-a}
b(q^t_1, w)_{Q,\tilde{E}} = (y_1-\tilde y_1)w(\r_1) + (\tilde y_2-y_2)w(\r_2) 
:= \sum_{j=1}^4 \sigma^t_{1,j}w(\r_j),
\end{align}
where $\sigma^t_{1,1} = (y_1-\tilde y_1),\,\sigma^t_{1,2} = (\tilde y_2-y_2)$ 
and $\sigma^t_{1,j} =0 $ for $j=3,4$. Using \eqref{q-norm-3}, we have
\begin{align}\label{q-norm-3-a}
b(q^n_1, w)_{Q,\tilde E} = (y_1-y_2)w(\r_1) + (y_2-y_1)w(\r_2) 
:= \sum_{j=1}^4 \sigma^n_{1,j}w(\r_j).
\end{align}
Therefore, 
\begin{align*}
b(q_E, w)_{Q,\tilde{E}} = \sum_{i=1}^4 \alpha_i b(q_i,w)_{Q,\tilde E} = 
\sum_{i=1}^4\sum_{j=1}^4\alpha_i\sigma_{i,j}w(\r_j).
\end{align*}
Moreover, due to assumption \ref{M2},
\begin{align*}
\sigma_{i,j} = \delta_{i,j} +\theta_{i,j}, 
\end{align*}
with $\theta_{i,j}=0$ if $\delta_{i,j}=0$ and $|\theta_{i,j}|\leq Ch^2$ otherwise. 
Indeed, consider, e.g., $i=j=1$, then, by \ref{M2},
\begin{align*}
|\sigma_{1,1} -\delta_{1,1}| = |\sigma^t_{1,1} - \delta^t_{1,1}|
= |(y_1-\tilde y_1) -(y_4-y_1 )| \leq Ch^2.
\end{align*}
Therefore, we obtain
\begin{align}
b(q_E, w)_{Q,\tilde{E}} 
&= \sum_{i=1}^4\sum_{j=1}^4\alpha_i\sigma_{i,j}w(\r_j) 
= b(q_E, w)_{Q,E} + \sum_{i=1}^4\sum_{j=1}^4\alpha_i\theta_{i,j}w(\r_j)
\nonumber \\
&\geq Ch^2 \sum_{j=1}^4 (w(\r_j))^2 +\sum_{i=1}^4\sum_{j=1}^4\alpha_i\theta_{i,j}w(\r_j), \label{aux-1}
\end{align}
using the first inequality in \eqref{local-inf-sup-const}, that 
$\mathbb{P}_h^Ew = 0$, and that
\begin{equation}\label{local-norm-equiv}
\|w\|_E^2 \sim h^2 \sum_{j=1}^4 (w(\r_j))^2,
\end{equation}
which follows from the norm equivalence $\|w\|_{E} \sim \|w\|_{Q,E}$ 
stated in Lemma~\ref{coercivity-lemma} and the shape regularity of the mesh. 

Finally, the second
inequality in \eqref{local-inf-sup-const} and a scaling argument 
imply that for every
$i=1,\dots, 4$ there exist constants $b_{i,k}, k=1,\dots,4$,
independent of $h$ such that
 \begin{align}\label{alpha-bound}
\alpha_i = h\sum_{k=1}^4 b_{i,k}w(\r_k).
\end{align}
	Then, there exists a constant $\tilde{C}$ independent of $h$ such that
\begin{align}\label{aux-2}
\left| \sum_{i=1}^4\sum_{j=1}^4\alpha_i\theta_{i,j}w(\r_j)\right| 
= \left| \sum_{i=1}^4 h \sum_{k=1}^4 b_{i,k}w(\r_k)
\sum_{j=1}^4 \theta_{i,j}w(\r_j) \right| 
\leq \tilde{C}h^3 \sum_{j=1}^4 (w(\r_j))^2.
	\end{align}
Combining \eqref{aux-1}--\eqref{aux-2} and taking $h \le C/\tilde{C}$,
we obtain \eqref{q_m_positive}:
\begin{align*}
b(q_E, w)_{Q,\tilde{E}} 
& \geq  Ch^2 \sum_{j=1}^4 (w(\r_j))^2 - \tilde{C}h^3 \sum_{j=1}^4 (w(\r_j))^2 
\geq (C-\tilde{C}h)h^2\sum_{j=1}^4 (w(\r_j))^2 \ge 0.
	\end{align*}

Next, consider the case of a parallelogram $E$ with parallelogram 
neighbors. In this case, 
\eqref{q-tan-1-a} and \eqref{q-tan-3-a} give
\begin{equation}\label{q-tan-1-3-a}
b(q^t_1, w)_{Q,E} = (y_4-y_1)(w(\r_1) - w(\r_2)), \quad
b(q^t_1, w)_{Q,\tilde{E}} = (y_1-\tilde y_1)(w(\r_1) - w(\r_2)).
\end{equation}
Similarly, \eqref{q-norm-1-a} and \eqref{q-norm-3-a} give
$$
b(q^n_1, w)_{Q,E} = (y_1-y_2)(w(\r_1) - w(\r_2)), \quad
b(q^n_1, w)_{Q,\tilde E} = (y_1-y_2)(w(\r_1) - w(\r_2)).
$$
Similar relationships hold for the rest of the basis functions. Therefore
there exist positive constants $c_i$, $i = 1,\dots,4$, such that 
$b(q_i,w)_{Q,\tilde E} = c_i b(q_i, w)_{Q,E}$. We can assume that 
$\alpha_i b(q_i,w)_{Q,E} \ge 0$ for $i = 1,\dots,4$, since, if 
$\alpha_i b(q_i,w)_{Q,E} < 0$, it can be omitted from the linear combination
$q_E = \sum_{i=1}^4 \alpha_i q_i$ and the resulting $q_E$ would still 
satisfy \eqref{local-inf-sup-const}. Therefore, \eqref{q_m_positive} holds:
$$
b(q_E,w)_{Q, \tilde E} = \sum _{i=1}^4 \alpha_i b(q_i,w)_{Q,\tilde E} 
= \sum _{i=1}^4 c_i\alpha_i b(q_i,w)_{Q,E} \ge 0.
$$ 

The assertion of the lemma now follows from \eqref{local-inf-sup-const}
and \eqref{q_m_positive}, where the second inequality in \eqref{macro-lemma-2-ineq}
follows from \eqref{alpha-bound}.
\end{proof}

We next note that the element norm equivalence \eqref{local-norm-equiv}
implies that for $w \in W_h$,
\begin{equation}\label{global-norm-equiv}
\|w\|^2 \sim h^2 \sum_{j=1}^{N_W} (w(\r_j))^2,
\end{equation}
where $N_W$ is the number of degrees of freedom of $W_h$. Therefore,
to prove \eqref{stokes-pair}, it is sufficient to control $h^2
(w(\r_j))^2$.  We will consider three sets of vertices and show that
each set can be controlled. Let 
\begin{align*}
& I_1 = \{j: \r_j \mbox{ is a vertex of a non-parallelogram} \}, \\
& I_2 = \{j: \mbox{all elements sharing } \r_j
\mbox{ are parallelograms and at least one has a non-parallelogram neighbor} \}, \\
& I_3 = \{j: \mbox{all elements sharing } \r_j
\mbox{ are parallelograms with parallelogram neighbors} \}.
\end{align*}
Clearly the union of the three sets covers all vertices of the mesh.

\begin{lemma}\label{lemma-set12}
If \ref{M1}--\ref{M2} hold, there exists a constant $C$ independent of 
$h$ such that for every $w \in W_h$, 
there exists $q \in Q_h$ such that 
\begin{equation}\label{I1-I2}
b(q,w)_{Q}  \geq C h^2 \sum_{j \in I_1 \cup I_2} (w(\r_j))^2, 
\quad \|q\|_1 \leq \|w\|.
\end{equation}
\end{lemma}
\begin{proof}
If $j \in I_1$, Lemma~\ref{macro-lemma-2} and \eqref{local-norm-equiv} 
imply that there exists $q_j \in Q^e_h(E)$ such that 
\begin{equation}\label{I1}
b(q_j,w)_{Q}  \geq C h^2 (w(\r_j))^2, \quad \|q_j\|_1 \leq \|w\|_E,
\end{equation}
where $E$ is the non-parallelogram element with vertex $\r_j$. 

Next, consider $j \in I_2$. Let $\r_k$ share an edge with
$\r_j$. Note that its two neighboring elements are
parallelograms. Denote them by $E$ and $\tilde E$ and let $q^t_1$ be the
tangential edge basis function. Using \eqref{q-tan-1-3-a}, we can take
$\tilde q_j = c h(w(\r_j) - w(\r_k))q^t_1$, which satisfies
\begin{equation}\label{Ej}
b(\tilde q_j,w)_{Q}  \geq C h^2 (w(\r_j) - w(\r_k))^2, 
\quad \|\tilde q_j\|_1 \leq \|w\|_E.
\end{equation}
Let $\r_k$ be the vertex that belongs to a non-parallelogram, denoted by $E_k$.
Then \eqref{I1} implies that there exists $q_k \in Q^e_h(E_k)$ such that
\begin{equation}\label{Ek}
b(q_k,w)_{Q}  \geq C h^2 (w(\r_k))^2, \quad \|q_k\|_1 \leq \|w\|_{E_k}.
\end{equation}
Let $q_j = \tilde q_j + q_k$. Due to \eqref{Ej} and \eqref{Ek}, $q_j$ satisfies
\begin{equation}\label{I2}
b(q_j,w)_{Q}  \geq C h^2 (w(\r_j))^2, \quad \|q_j\|_1 \leq \|w\|_{E \cup E_k}.
\end{equation}
Finally, $q \in Q_h$ defined as the sum of the functions constructed 
in \eqref{I1} and \eqref{I2} satisfies \eqref{I1-I2}.
\end{proof}

We now consider the set of vertices $I_3$. If $\r_j$ and $\r_k$ are
two vertices in the set that share an edge, \eqref{Ej} implies that if
one of them is controlled, then so is the other. Therefore it is enough
to consider a subset of vertices that do not share an edge, which we
denote by $\tilde I_3$. For each vertex $\r_j$, let $M_j$ be the union
of elements that share $\r_j$. We note that the set $S = \{M_j: j \in \tilde I_3\}$
is non-overlapping. Let $\bar M = \cup_{j \in \tilde I_3} M_j$.
For $M \in S$, let $Q^e_h(M)$
be the span of all edge degrees of freedom of $Q_h(M)$ and let
\begin{align*}
N_M = \{w \in W_h(M) :\, b(q,w)_{Q,M} = 0,\, \forall q\in Q^e_h(M) \}.
\end{align*}
Recall that all elements in $M$ are parallelograms. The argument of
Lemma~\ref{N_M-lemma} can be easily extended to show that $N_M$
consists of constant functions. For $M \in S$, let 
$\mathbb{P}_h^M$ denote the $L^2$-projection from $W_h(M)$ onto $N_M$.
Since the neighbors of all elements in $M$ are parallelograms, 
Lemma~\ref{macro-lemma-2} implies that for any $w \in W_h$, 
there exists $q_M \in Q^e_h(M)$ satisfying
\begin{align}\label{macro-lemma-2-ineq-M}
b(q_M,w)_{Q}  \geq C_1 \|(I-\mathbb{P}_h^M)w\|_M^2 \quad \mbox{and} 
\quad \|q_M\|_1 \leq \|(I-\mathbb{P}_h^M)w\|_M.
\end{align}
Let
\begin{align*}
M_h =\{\mu \in L^2(\Omega): \mu\big|_M \in N_M, \,\, \forall M \in S,
\, \mu = 0 \mbox{ otherwise}\}.
\end{align*}
Let $\mathbb{P}_h$ be the $L^2$-projection from $W_h$ onto $M_h$. Then
\eqref{macro-lemma-2-ineq-M} implies that for any $w \in W_h$, there 
exists $\tilde q \in Q_h$ satisfying
\begin{align}\label{macro-lemma-2-ineq-global}
b(\tilde q,w)_{Q}  \geq C_1 \|(I-\mathbb{P}_h)w\|_{\bar M}^2 \quad \mbox{and} 
\quad \|\tilde q\|_1 \leq \|(I-\mathbb{P}_h)w\|_{\bar M}.
\end{align}
The next lemma shows that $\mathbb{P}_h w$ can also be controlled.
\begin{lemma}\label{macro-lemma-3}
If \ref{M1} holds, there exists a constant $C_2>0$ such 
that for every $w \in W_h$ there exists $g \in Q_h$ such that
\begin{align*}
b(g,\mathbb{P}_h w)_{Q} = \|\mathbb{P}_h w\|_{\bar M}^2\quad \text{ and }\quad 
\|g\|_1 \leq C_2\|\mathbb{P}_h w\|_{\bar M}.
\end{align*}
\end{lemma}
\begin{proof}
	Let $w\in W_h$ be arbitrary. Since $\mathbb{P}_h w\in L^2(\O)$, there exists $z \in H^1(\Omega)$ such that
	\begin{align*}
	\dvr z = \mathbb{P}_h w \quad 
\mbox{and}\quad \|z\|_1 \leq C\|\mathbb{P}_h w\|_{\bar M}.
	\end{align*}
Following \cite[Lemma 3.3]{stenberg1984analysis}, there exists an operator 
$I_h:H^1(\Omega) \rightarrow \tilde Q_h$ such that
\begin{align*}
(\dvr z,\mu) &=b(I_h z, \mu),\quad \forall \mu \in M_h, \quad 
\mbox{and} \quad \|I_h z\|_1 \leq C\|z\|_1,
\end{align*}
where $\tilde Q_h$ is the subspace of $Q_h$ consisting of element-wise mapped
bilinear functions. We note that the argument in
\cite[Lemma~3.3]{stenberg1984analysis} requires that the interfaces
between macroelements have at least two edges. Recall that our
macroelements consist of all parallelograms sharing a vertex and their
neighbors are also parallelograms. We can therefore choose the subset
$\tilde I_3$ appropriately to satisfy this requirement. Here we also
consider $\Omega \setminus \bar M$ as one macroelement. 

We next note that for $q \in \tilde Q_h$ and $\mu \in M_h$, on any $E \in \Tc_h$, 
$$
b(q, \mu)_{E} = \int_{\hat E} \tr(\DF^{-T}_{E}\hat{\nabla}\hat{q})\hat{\mu}J_E 
\, d\hat\x.
$$
A direct calculation shows that the integrated quantity on $\hat E$ is
bilinear, and hence, using that the quadrature rule is exact for bilinears, 
$b(I_h z, \mu) = b(I_h z, \mu)_Q$. The proof is completed by taking $g = I_h z$. 
\end{proof}

\begin{lemma}\label{lemma-set3}
If \ref{M1} holds, there exists a constant $C$ independent of 
$h$ such that for every $w \in W_h$, 
there exists $q \in Q_h$ such that 
\begin{equation}\label{I3}
b(q,w)_{Q}  \geq C h^2 \sum_{j \in I_3} (w(\r_j))^2, 
\quad \|q\|_1 \leq \|w\|.
\end{equation}
\end{lemma}
\begin{proof}
Let $w \in W_h$ be given, and let $\tilde q \in Q_h,\,g\in Q_h,\, C_1$ and $C_2$ 
be as in \eqref{macro-lemma-2-ineq-global} and Lemma \ref{macro-lemma-3}. 
Set $q = \tilde q + \delta g$, where $\delta = 2C_1(1+C_2^2)^{-1}$. We then have
\begin{align*}
b(q,w)_{Q} &=b(\tilde q,w)_{Q}+\delta b(g,w)_{Q}  
= b(\tilde q,w)_{Q}+\delta b(g,\mathbb{P}_h w)_{Q} 
+ \delta b(g,(I-\mathbb{P}_h) w)_{Q} \\
&\geq C_1\|(I-\mathbb{P}_h)w\|_{\bar M}^2 +\delta \|\mathbb{P}_h w\|_{\bar M}^2
-\delta \|g\|_1\|(I-\mathbb{P}_h)w\|_{\bar M} \geq C_1(1+C_2^2)^{-1}\|w\|_{\bar M}^2,
\end{align*}
and 
$\|q\|_1 \leq \|(I-\mathbb{P}_h)w\|_{\bar M} 
+\delta C_2\|\mathbb{P}_h w\|_{\bar M} \leq C \|w\|_{\bar M}$.
The assertion of the lemma follows from \eqref{local-norm-equiv}. 
\end{proof}

We are now ready to prove the main result stated in Theorem \ref{macro-inf-sup}:
\begin{proof}[Proof of Theorem \ref{macro-inf-sup}] 
The assertion of the theorem follows from Lemma~\ref{lemma-set12}, 
Lemma~\ref{lemma-set3}, and \eqref{global-norm-equiv}.
\end{proof}

We conclude with the solvability result for the MSMFE-1 method \eqref{h-weak-P1-1}-\eqref{h-weak-P1-3}.
\begin{theorem}
Under the assumptions \ref{M1}--\ref{M2}, there exists a unique solution of 
\eqref{h-weak-P1-1}-\eqref{h-weak-P1-3}. 
\end{theorem}
\begin{proof}
The existence and uniqueness of a solution to 
\eqref{h-weak-P1-1}-\eqref{h-weak-P1-3} follows from 
\ref{S3-P1} and \ref{S4-P1}. Lemma~\ref{coercivity-lemma} implies the coercivity 
condition \ref{S3-P1}. 
Assuming \ref{M1}--\ref{M2}, the inf-sup condition \ref{S4-P1} follows from a 
combination of 
Theorem~\ref{inf-sup-P1-proof}, Lemma~\ref{aux-lemma}, and 
Theorem~\ref{macro-inf-sup}.
\end{proof}

\subsection{Reduction to a cell-centered displacement system of the MSMFE-1 method}
The algebraic system that arises from
\eqref{h-weak-P1-1}--\eqref{h-weak-P1-3} is of the form
\eqref{sp-matrix}, where the matrix $\Ag$ is different from the one in
the MSMFE-0 method, due the the quadrature rule, i.e., $(\Ag)_{ij} =
(\tau_j,\xi_i)_Q$. As in the MSMFE-0 method, the quadrature rule in
$(A\sigma_h,\tau)_Q$ in \eqref{h-weak-P1-1} localizes the basis
functions interaction around vertices, so the matrix $A_{\s\s}$ is
block diagonal with $2k\times 2k$ blocks, where $k$ is the number of
elements that share a vertex.  The stress can be eliminated, resulting
in the displacement-rotation system \eqref{msmfe0-system}. The matrix
in \eqref{msmfe0-system} is symmetric and positive definite, due to
\eqref{spd-matrix} and the inf-sup condition \ref{S4-P1}.

Furthermore, the quadrature rule in the stress-rotation bilinear forms
$(\g_h,\tau)_Q$ and $(\sigma_h,\xi)_Q$ also localizes the interaction
around vertices, since there is one rotation basis function associated
with each vertex. Therefore the matrix $\Ag$ is block-diagonal with
$1\times 2k$ blocks, resulting in a diagonal rotation 
matrix $\Ag\As^{-1}\Ag^T$. As a result, the rotation $\gamma_h$ can be 
trivially eliminated from \eqref{msmfe0-system}, leading to the cell-centered
displacement system
\begin{align}
\left( A_{\s u}A^{-1}_{\s\s}A^T_{\s u} - A_{\s u} A^{-1}_{\s\s}A^T_{\s\g}(A_{\s\g}A^{-1}_{\s\s}A^T_{\s\g})^{-1}A_{\s\g}A^{-1}_{\s\s}A^T_{\s u} \right) u = \hat{f} \label{disp-syst}.
\end{align}
The above matrix is symmetric and positive definite, since it is a
Schur complement of the symmetric and positive definite matrix in 
\eqref{msmfe0-system}, see \cite[Theorem 7.7.6]{Horn-Johnson}.


\section{Error estimates}
In this section we establish optimal convergence for all variables, as
well as the superconvergence for the displacement. We start by
providing several results that will be used in the analysis.

\subsection{Preliminaries}
For the rest of the paper we assume that the quadrilateral elements
are $O(h^2)$-perturbations of parallelograms known as $h^2$-parallelograms. 
In particular, with the notation from Figure~\ref{elements}, we assume that
\begin{align}\label{h2-parall}
\|\r_{34}-\r_{21}\| \leq Ch^2.
\end{align}
Elements of this type are obtained by uniform refinements of a general
quadrilateral grid or if the mesh is obtained by a smooth map. This is
a standard assumption for the symmetric multipoint flux approximation
method \cite{wheeler2006multipoint}, required due to the reduced
approximation properties of the $\BDM_1$ space on general
quadrilaterals \cite{arnold2005quadrilateral}. If \eqref{h2-parall} holds,
it is easy to check that 
\begin{align}
|\DF_E|_{1,\infty,\Eh} \leq Ch^2 \quad \text{and} \quad \left|\frac{1}{J_E}\DF_E\right|_{j,\infty,\Eh}\leq Ch^{j-1},\, j=1,2. \label{scaling-of-mapping-2}
\end{align}

In the analysis we will utilize several projection operators. It is
known \cite{brezzi1985two, brezzi1991mixed, wang1994mixed} that there
exists a projection operator $\Pi: \X \cap H^1(\Omega,\M) \to \X_h$ such that
\begin{align}
    (\dvr (\Pi\t -\t), v) = 0,  \quad \forall \, v \in V_h. 
\label{bdm-interpolant1}
\end{align}
The operator $\Pi$ is defined locally on an element $E$ by 
\begin{equation}\label{Pi-Piola}
\Pi \t \overset{\Pc}{\leftrightarrow} \hat\Pi\hat\t,
\end{equation}
where $\hat\Pi$ is a reference element interpolant. 
We will also utilize the lowest order Raviart-Thomas space 
\cite{raviart1977mixed, brezzi1991mixed}:
$\Xh^{\RT}(\Eh) = \begin{pmatrix} \alpha_1 + \beta_1\xh \\ 
\alpha_2 + \beta_2\yh \end{pmatrix} \times 
\begin{pmatrix} \alpha_3 + \beta_3\xh \\ 
\alpha_4 + \beta_4\yh \end{pmatrix}$.  
The degrees of freedom of $\Xh^{\RT}(\Eh)$ are the values of the normal
components at the midpoints of the edges. A projection operator 
$\Pi^{\RT}$ onto $\X^{\RT}_h$ similar to \eqref{bdm-interpolant1} exists 
\cite{raviart1977mixed, brezzi1991mixed}, which satisfies for any edge $e$,
\begin{equation}\label{Pi-RT-orth}
\langle (\Pi^{\RT}\tau - \tau) n_e, \chi n_e\rangle_e = 0, 
\quad \forall \, \chi \in \X^{\RT}_h.
\end{equation}
It is also easy to see that $\Pi^{\RT}$ satisfies
\begin{equation}
    \dvr \t = \dvr \Pi^{\RT} \t, \quad \forall \t \in \X_h 
\label{rt-operator1}
\end{equation}
and
\begin{equation}
    \|\Pi^{\RT}\t\| \le C \|\t\|, \quad \forall \t \in \X_h.  \label{rt-operator2}
\end{equation}
Let $Q^u_h$ be a projection operator
onto $V_h$ satisfying for any $v \in L^2(\Omega,\R^2)$, 
\begin{equation}
(\hat Q^u \hat v - \hat v, \hat w)_{\hat E} = 0, \quad \forall \hat w \in \hat V(\hat E),
\qquad  Q^u_h v = \hat Q^u \hat v \circ F_E^{-1} \,\, \forall E \in \Tc_h.
\end{equation}
It follows from \eqref{prop-piola} that 
\begin{equation}\label{disp-proj}
(Q^u_h v - v, \dvr \tau) = 0, \quad \forall \, \tau \in \X_h.
\end{equation}
Let $Q^{\g}_h$ be the $L^2$-orthogonal 
projection operator onto $\W_h$
satisfying for any $\xi \in L^2(\Omega,\N)$, 
\begin{equation}
(Q^{\g}_h \xi - \xi, \zeta) = 0, \qquad \forall \zeta \in \W_h^k.
\end{equation}
The next lemma summarizes the well-known approximation properties of the above
operators.
\begin{lemma}
There exists a constant $C$ independent of $h$ such that
\begin{align}
& \|v - Q^u_h  v\| \le C\|v\|_r h^r, && \forall v\in H^r(\O,\R^2),
&& 0\le r\le 1, \label{approx-1}\\
& \|\xi - Q^{\g}_h \xi\| \le C\|\xi\|_r h^r, && \forall \xi \in H^r(\O,\N),
&& 0\le r\le 1, \label{approx-2} \\
& \|\t - \Pi\t\| \le C \|\t\|_r h^r, && \forall \t\in H^r(\O,\M),\, 
&& 1\le r\le 2, \label{approx-3}\\
& \|\t - \Pi^{\RT}\t\| \le C \|\t\|_r h^r, && \forall \t\in H^{r}(\O,\M), 
&& 0\le r\le 1, \label{approx-4} \\
& \|\dvr(\t - \Pi\t) \|+\| \dvr(\t - \Pi^{\RT}\t) \| \le C \|\dvr\t\|_r h^r, 
&& \forall \t\in H^{r+1}(\O,\M), && 0\le r\le 1. \label{approx-5}
\end{align}
\end{lemma}
\begin{proof}
Estimates \eqref{approx-1} and \eqref{approx-2} can be found in
\cite{ciarlet2002finite}. Estimates
\eqref{approx-3}--\eqref{approx-5} are proved in
\cite{arnold2005quadrilateral,wang1994mixed}.
\end{proof}
	
We note that on general quadrilaterals, \eqref{approx-1},
\eqref{approx-2} and \eqref{approx-4} are also valid, while
\eqref{approx-3} and \eqref{approx-5} hold only with $r=1$ and $r=0$,
respectively.

\begin{corollary}
There exists a constant $C$ independent of $h$ such that for all $E \in \Tc_h$,
\begin{align}
& \|\Pi \t\|_{j,E} \leq C \|\t\|_{j,E}, && \forall \t \in H^j(E,\M), 
\quad j=1,2, \label{h1-continuity-bdm} \\
& \|\Pi^{\RT} \t\|_{1,E} \leq C \|\t\|_{1,E}, && \forall \t \in H^1(E,\M), 
\label{h1-continuity-rt} \\
& \|Q^{\g}_h \xi\|_{j,E} \leq C \|\xi\|_{j,E}, && \forall \xi \in H^1(E,\N), 
\quad j=1,2.
\label{h1-continuity-l2}
\end{align}
\end{corollary}
\begin{proof}
The proof follows from the approximation properties 
\eqref{approx-2}--\eqref{approx-4} and the use of the inverse inequality, 
see e.g., \cite[Lemma~5.1]{msmfe_simp}.
\end{proof}
We remind the reader that stress tensors are mapped from the
reference element via the Piola transformation, while displacements
and rotations are mapped using standard change of variables, see \eqref{maps}. 
The following results can be found in
\cite{wheeler2006multipoint}, where 
$\t \overset{\Pc}{\leftrightarrow} \hat{\t}$:
\begin{align}
& |\th|_{j,\Eh} \le Ch^j\|\t\|_{j,E}, \quad j\geq0,
\label{scaling-1} \\
& \inp[\th - \hat{\Pi}^{\RT}\th]{\ch_0}_{\hat{Q}, \Eh} = 0 \quad 
\mbox{for all constant tensors $\ch_0$}, \label{lemma-orth-quadrature}\\
& |(A\Pi\s, \t - \Pi^{\RT} \t)_{Q,E} | \le C h\|\s\|_{1,E}\|\t\|_{E}
\quad \forall \, \t \in \X_h. \label{lemma-orth-quad}
\end{align}
Also, for $\xi \leftrightarrow \hat{\xi}$, using standard change of variables,
\begin{equation}\label{standard-cov}
|\hat\xi|_{j,\Eh} \le Ch^{j-1}\|\xi\|_{j,E}, 
\quad |\hat\xi|_{j,\infty,\Eh} \le Ch^{j}\|\xi\|_{j,\infty,E},
\quad j\geq0.
\end{equation}
For $\tau, \chi \in \X_h$, $\xi \in \W^1_h$,
denote the element quadrature errors by
\begin{align*}
    \tet(A\tau,\chi)_E \equiv (A\tau,\chi)_E - (A\tau,\chi)_{Q,E}, \quad
    \del(\tau,\xi)_E \equiv (\tau,\xi)_E - (\tau,\xi)_{Q,E},
\end{align*}
and define the global quadrature errors by $\theta(A\tau,\chi)|_E =
\tet(A\tau,\chi)_E$, $\del(\tau,\xi)|_E = \del(\tau,\xi)_E$. Similarly
denote the quadrature errors on the reference element by
$\teth(\cdot,\cdot)$ and $\delh(\cdot,\cdot)$.

Denote $A \in W^{j,\infty}_{\Tc_h}$ if
$A \in W^{j,\infty}(E) \, \forall E \in \Tc_h$ and $\|A\|_{j,\infty,E}$ is uniformly bounded independently of $h$.
\begin{lemma}\label{lemma-theta-bound}
If $A \in W^{1,\infty}_{\Tc_h}$, there exists a constant C independent of $h$ 
such that $\forall \, \t \in \X_h$, $\chi \in \X_h^{\RT}$,
	\begin{equation}
	| \theta(A\t,\chi) | \le C \sum_{E\in\Tc_h} h\|A\|_{1,\infty,E}\|\t\|_{1,E}\|\chi\|_{E}. \label{theta-bound-1}
	\end{equation}
Moreover, there exist a constant C independent of $h$ such that for 
all $\t \in \X^{\RT}_h$ and  $\xi \in \W^1_h$,
	\begin{align}
	| \delta(\t,\xi) | \le C \sum_{E\in\Tc_h} h\|\t\|_{1,E}\|\xi\|_{E}, \label{theta-bound-2} 	\\
	| \delta(\t,\xi) | \le C \sum_{E\in\Tc_h} h\|\t\|_{E}\|\xi\|_{1,E}. \label{theta-bound-3}
	\end{align}
\end{lemma}
\begin{proof}
For a function $\varphi$ defined on $\Eh$, let $\bar\varphi$ be its mean value.
We have
\begin{align}
\theta_E(A\t,\chi) & = 
\hat{\theta}_{\Eh}(\Ah \th \frac{1}{J_E} \DF^T_E, \ch \DF^T_E)
\nonumber \\
& = \hat{\theta}_{\Eh}((\Ah - \bar\Ah) \th \frac{1}{J_E} \DF^T_E, \ch \DF^T_E)
+ \hat{\theta}_{\Eh}(\bar\Ah \th (\frac{1}{J_E} \DF^T_E 
- \overline{\frac{1}{J_E} \DF^T_E}) , \ch \DF^T_E) 
\nonumber \\
& \quad + \hat{\theta}_{\Eh}(\bar\Ah \th \overline{\frac{1}{J_E} \DF^T_E} , 
\ch (\DF^T_E - \overline{\DF^T_E})) 
+ \hat{\theta}_{\Eh}(\bar\Ah \th \overline{\frac{1}{J_E} \DF^T_E} , 
\ch \overline{\DF^T_E}) \equiv \sum_{k=1}^4 I_k.
\label{theta-bound}
\end{align}
Using the Bramble-Hilbert lemma \cite{ciarlet2002finite}, 
\eqref{scaling-of-mapping}, \eqref{scaling-1}, and \eqref{standard-cov},
we bound the first term on the right in \eqref{theta-bound} as follows:
\begin{align}
|I_1| \leq C|\Ah|_{1,\infty, \Eh} \|\th\|_{\Eh}\|\ch\|_{\Eh} \leq Ch\|A\|_{1,\infty,E}\|\t\|_{E}\|\chi\|_{E}. \label{I1-bound}
\end{align} 
Similarly, using \eqref{scaling-of-mapping}, 
\eqref{scaling-of-mapping-2}, \eqref{scaling-1}, and \eqref{standard-cov},
\begin{equation}\label{bound-I2-I3}
|I_2| + |I_3| \le C h \|\Ah\|_{0,\infty, \Eh} \|\th\|_{\Eh}\|\ch\|_{\Eh} 
\leq Ch\|A\|_{0,\infty,E}\|\t\|_{E}\|\chi\|_{E}.
\end{equation}
To bound $I_4$, recall that the trapezoidal quadrature rule is exact for bilinear
functions. Since $\ch \in \Xh^{\RT}(\Eh)$ is linear,
$I_4 = 0$ for any constant tensor
$\th$. Using
the Bramble-Hilbert lemma, \eqref{scaling-of-mapping}, \eqref{scaling-1},  
and \eqref{standard-cov}, we have
\begin{align}
|I_4| \leq C \|\Ah\|_{0,\infty,\Eh} |\th|_{1,\Eh}\|\ch\|_{\Eh} 
\leq Ch \|A\|_{0,\infty, E}\|\t\|_{1,E}\|\chi\|_E.
\label{bound-I4}
\end{align}
Combining \eqref{theta-bound}--\eqref{bound-I4}
and summing over the elements implies 
\eqref{theta-bound-1}. Similarly, using the exactness of the 
quadrature rule for bilinears, the Bramble-Hilbert lemma,
\eqref{scaling-of-mapping}, 
\eqref{scaling-of-mapping-2}, \eqref{scaling-1}, and \eqref{standard-cov},
we have
\begin{align*}	
|\delta_E(\t,\xi)| & = |\delh(\th\DF_E^T,\hat{\xi})| 
\le |\delh(\th(\DF_E^T - \overline{\DF_E^T}),\hat{\xi})| 
+ |\delh(\th\overline{\DF_E^T},\hat{\xi})|\\
&\leq C\left(|\DF_E|_{1,\infty,\Eh}\|\th\|_{\Eh}\|\hat{\xi}\|_{\Eh} 
+ \|\DF_E\|_{0,\infty,\Eh}|\th|_{1,\Eh}\|\hat{\xi}\|_{\Eh}\right)
\leq Ch\|\t\|_{1,E}\|\xi\|_E,
\end{align*}
which implies \eqref{theta-bound-2}. Bound \eqref{theta-bound-3}
follows in a similar way.
\end{proof}

\begin{lemma} \label{orth-rot-quad}
There exists a constant C independent of $h$ such that for all $\t \in \X_h$
and $\xi \in \W_h^1$,
\begin{equation}
|(\t-\Pi^{\RT}\t, \xi)_Q| \leq Ch\|\t\|\|\xi\|_1.
\end{equation}
\end{lemma}
\begin{proof}
The proof follows from mapping to the reference element and using \eqref{lemma-orth-quadrature}. 
\end{proof}

\subsection{First order convergence for all variables}
The convergence analysis presented below is different from the one on
simplices from \cite{msmfe_simp}. In particular, since the quadrature
error bounds \eqref{theta-bound-1}--\eqref{theta-bound-3} require that
one of the arguments is in $\X^{\RT}_h$, rather than $\X_h$, the error
equations need to be manipulated in a special way.

\begin{theorem}
Let $A\in W^{1,\infty}_{\Tc_h}$.
For the solution $(\s,u,\g)$ of \eqref{weak-1}--\eqref{weak-3} 
and its numerical approximation
$(\s_h,u_h,\g_h)$ obtained by either the MSMFE-0 method 
\eqref{h-weak-P0-1}--\eqref{h-weak-P0-3}
or the MSMFE-1 method \eqref{h-weak-P1-1}--\eqref{h-weak-P1-3}, there exists a
constant $C$ independent of $h$ such that
\begin{align}
\|\s-\s_h\|_{\dvr}+\| u -u_h \|+\| \g-\g_h\| \leq Ch(\|\s\|_1 + \|\dvrg \s\|_1+\|u\|_1 + \|\g\|_1). \label{msmfe-error} 		
\end{align}
\end{theorem}
\begin{proof}
We present the argument for the MSMFE-1 method, which includes the
proof for the MSMFE-0 method.
We form the error system by subtracting the MSMFE-1 method 
\eqref{h-weak-P1-1}--\eqref{h-weak-P1-3} from \eqref{weak-1}--\eqref{weak-3}:
\begin{align}
(A\s ,\t)-(A \s_h,\t)_Q + (u-u_h,\dvr\t) + (\g, \t)-(\g_h,\t)_Q 
& = \gnp[g-\Pc_0 g]{\t n}_{\Gamma_D},
&& \t \in \X_h, \label{P1-error-eq-1}\\
(\dvr(\s - \s_h), v) &= 0, && v \in V_h, \label{P1-error-eq-2}\\
(\s,\xi)-(\s_h, \xi)_Q &= 0, && \xi \in \W^1_h. \label{P1-error-eq-3}
\end{align}
Using \eqref{disp-proj}, \eqref{Pi-RT-orth}, and \eqref{rhs-proj},  
we rewrite the first error equation as 
\begin{align}\label{error-P1-eq1}
& (A(\Pi\s- \s_h),\t)_Q + (Q^u_h  u-u_h,\dvr\t) \nonumber \\
& \qquad = -(A\s ,\t) + (A\Pi\s ,\t)_Q
- (\g, \t)+(\g_h,\t)_Q + \gnp[g]{(\t-\Pi^{\RT}\t) n}_{\Gamma_D}.
\end{align}
For the first two terms on the right above we write
\begin{align}\label{error-P1-I1-eq1}
& -(A\s ,\t) + (A\Pi\s ,\t)_Q =
- (A\s,\t - \Pi^{\RT}\t)  
- (A(\s - \Pi\s),\Pi^{\RT}\t) \nonumber \\ 
& \qquad\quad - (A\Pi\s,\Pi^{\RT}\t) 
+ (A\Pi\s,\Pi^{\RT}\t)_Q 
+ (A\Pi\s,\t - \Pi^{\RT}\t)_Q.
\end{align}
The second two terms on the right in \eqref{error-P1-eq1} can be rewritten as
\begin{align}\label{error-P1-I2-eq1}
& - (\g, \t)+(\g_h,\t)_Q = 
-(\g,\t-\Pi^{\RT}\t) -(\g-Q^{\g}_h\g,\Pi^{\RT}\t) \nonumber \\
& \qquad - (\Pi^{\RT}\t,Q^{\g}_h\g) + (\Pi^{\RT}\t,Q^{\g}_h\g)_Q
+(Q^{\g}_h\g, \t-\Pi^{\RT}\t)_Q +(\g_h-Q^{\g}_h\g,\t)_Q.
\end{align}
Combining the first terms in \eqref{error-P1-I1-eq1} and 
\eqref{error-P1-I2-eq1} with the last term in \eqref{error-P1-eq1} gives
\begin{equation}\label{zero-terms}
- (A\s,\t - \Pi^{\RT}\t) -(\g,\t-\Pi^{\RT}\t) 
+ \gnp[g]{(\t-\Pi^{\RT}\t) n}_{\Gamma_D} = 0,
\end{equation}
which follows from testing \eqref{weak-1} with $ \t - \Pi^{\RT}\t$ and 
using \eqref{rt-operator1}. The rest of the terms in \eqref{error-P1-I1-eq1} and 
\eqref{error-P1-I2-eq1} are bounded as follows. Using \eqref{approx-3}
and \eqref{rt-operator2}, we have
\begin{align}
|(A(\s - \Pi\s),\Pi^{\RT}\t)| \leq Ch\|\s\|_1\|\t\| 
\leq Ch^2\|\s\|^2_1 + \epsilon \|\t\|^2. \label{error-P1-I1-eq4}
\end{align}
For the third and fourth terms on the right in \eqref{error-P1-I1-eq1},
using \eqref{theta-bound-1}, 
\eqref{h1-continuity-bdm} and \eqref{h1-continuity-rt}, we obtain
	\begin{align}
	|\theta(A\Pi\s,\Pi^{\RT} \t)| \leq Ch\|\s\|_1\|\t\| \leq Ch^2\|\s\|^2_1+ \epsilon \|\t\|^2.\label{error-P1-I1-eq2}
	\end{align}
Using \eqref{lemma-orth-quad}, we write
	\begin{align}
	|(A\Pi\s,\t - \Pi^{\RT}\t)_Q| \leq Ch\|\s\|_1\|\t\| \leq Ch^2\|\s\|^2_1 + \epsilon\|\t\|^2. \label{error-P1-I1-eq3}
	\end{align} 
We next bound the terms on the right in \eqref{error-P1-I2-eq1}.
Due to \eqref{approx-2} and \eqref{rt-operator2}, we have
	\begin{align}
	 |(\g-Q^{\g}_h \g,\Pi^{\RT} \t)| \leq Ch\|\g\|_1\|\t\| 
\leq Ch^2\|\g\|^2_1 + \epsilon\|\t\|^2. \label{error-P1-eq4}
	\end{align}
Using \eqref{theta-bound-3}, \eqref{rt-operator2}, and \eqref{h1-continuity-l2},
we have 
\begin{align}
	|\delta(\Pi^{\RT} \t,Q^{\g}_h \g)|  \leq Ch\|\t\|\| \g\|_{1}\leq Ch^2\|\g\|^2_1 +\epsilon\|\t\|^2.  \label{error-P1-eq5}
	\end{align}
Using Lemma \ref{orth-rot-quad}, we obtain
	\begin{align}
	|( Q^{\g}_h \g,\t- \Pi^{\RT}\t)_Q| \leq Ch\|\g\|_1\|\t\| \leq Ch^2\|\g\|^2_1 +\epsilon\|\t\|^2.  \label{error-P1-eq6}
	\end{align}
Combining \eqref{error-P1-eq1}--\eqref{error-P1-eq6}, we obtain
	\begin{align}
	(A(\Pi\s- \s_h),\t)_Q +& (Q^u_h  u-u_h,\dvr\t)  \leq Ch^2(\|\s\|_1^2+\|\g\|^2_1) + \epsilon\|\t\|^2 +(\g_h-Q^{\g}_h\g,\t)_Q. \label{error-P1-eq8}
	\end{align}
We next note that, using \eqref{bdm-interpolant1}, the second error 
equation \eqref{P1-error-eq-2} implies that 
\begin{align}
\dvr(\Pi\s - \s_h) = 0. \label{error-P1-eq9-1}
\end{align}
The third error equation \eqref{P1-error-eq-3} implies
\begin{align}
(\Pi\s-\s_h,\xi)_Q & = (\Pi\s-\s,\xi)_Q
+ (\s - \Pi^{\RT}\s,\xi)_Q
-\delta(\Pi^{\RT}\s,\xi) + (\Pi^{\RT}\s - \s,\xi) \nonumber \\
& \le C h^2\|\s\|_1^2 + \epsilon\|\xi\|^2, \label{error-P1-eq9-2}
\end{align}
using \eqref{approx-3}, \eqref{approx-4}, \eqref{theta-bound-2}, and
\eqref{h1-continuity-rt} for the inequality. We now set $\t = \Pi\s -
\s_h$ in \eqref{error-P1-eq8}, $\xi = \g_h-Q^{\g}_h\g$ in 
\eqref{error-P1-eq9-2}, use
\eqref{coercivity} and \eqref{error-P1-eq9-1}, and take $\epsilon$ small enough 
to obtain
\begin{align}
\|\Pi\s - \s_h\|^2 \leq Ch^2(\|\s\|_1^2+\|\g\|^2_1)  +\epsilon\|\g_h - Q^{\g}_h \g\|^2. \label{error-P1-eq10}
\end{align}
We apply the inf-sup condition \eqref{inf-sup-P1} 
to $(Q^u_h u-u_h,Q^{\g}_h\g-\g_h) \in V_h\times \W^1_h$ and use 
\eqref{P1-error-eq-1} to obtain
\begin{align}
\|Q^u_h u -u_h\| + \|Q^{\g}_h \g-\g_h\| & \leq C\sup\limits_{\t\in \X_h}
\frac{
-(A\s,\t)+(A\s_h,\t)_Q-(\g,\t)+(Q^{\g}_h\g,\t)_Q
+ \langle g-\Pc_0g,\t n\rangle_{\Gamma_D}}{\|\t\|_{\dvr}} \nonumber \\
& \le C( h\|\s\|_1 + h\|\g\|_1 + \|\Pi\s - \s_h\|), 
\label{error-P1-eq11}
\end{align}
where the numerator terms have been bounded in a manner similar to the bounds
for the terms in the error equation \eqref{P1-error-eq-1} presented above.
Next, we combine a sufficiently small multiple of 
\eqref{error-P1-eq11} with \eqref{error-P1-eq10}, and choose 
$\epsilon$ in \eqref{error-P1-eq10} small enough to get
\begin{align}
	 \|\Pi\s-\s_h\| + \|Q^u_h u -u_h \| + \|Q^{\g}_h \g-\g_h\| 
\leq Ch(\|\s\|_1+\|\g\|_1). \label{error-P1-final1}
\end{align}
The assertion of the theorem follows from \eqref{error-P1-final1},
\eqref{error-P1-eq9-1}, \eqref{approx-1}--\eqref{approx-3}, and
\eqref{approx-5}. The proof for the MSMFE-0 method can obtained by
omitting the quadrature error terms $\delta(\cdot,\cdot)$.
\end{proof}

\subsection{Second order convergence for the displacement}
We next present superconvergence analysis for the displacement using a duality
argument. We need the following improved bounds on the quadrature errors.
\begin{lemma}
If $A\in W^{2,\infty}_{\Tc_h}$, there exists a constant $C$
independent of $h$ such that for all $\t \in \X_h$ and $\chi \in
\X^{\RT}_h$
\begin{align}
|\theta(A\t,\chi)| \leq C\sum_{E\in \Tc_h} h^2 \|\t\|_{2,E}\|\chi\|_{1,E}. 
\label{h2-theta}
\end{align}
For all $\chi \in \X_h^{\RT},\,\xi\in \W^1_h$ there exists a constant $C$ 
independent of $h$ such that
\begin{align}
|\delta(\chi,\xi)| \leq C\sum_{E\in \Tc_h} h^2 \|\chi\|_{1,E}\|\xi\|_{2,E}.	
\label{h2-delta}
\end{align}
\end{lemma}
\begin{proof}
The proof of \eqref{h2-theta} is given in
\cite[Lemma~4.2]{wheeler2006multipoint}. It uses the Piano kernel theorem
\cite[Theorem 5.2-3]{stroud1971approximate} and the fact that the quadrature 
rule is exact for bilinear functions. The proof of \eqref{h2-delta} is similar.
\end{proof}

We consider the auxiliary elasticity problem: find $\phi$ and $\psi$ such that
\begin{align}
\begin{aligned}
& \psi = A^{-1}\epsilon(\phi), \quad \dvrg \psi = (Q^u_h  u - u_h)  
\quad \mbox{in } \O, \\
& \phi = 0 \mbox{ on } \Gd, \quad \psi\,n = 0 \mbox{ on } \Gn. \label{aux}
\end{aligned}
\end{align}
We assume that the above problem is $H^2$-elliptic regular, 
see \cite{grisvard2011elliptic} for sufficient conditions:
\begin{align}
\|\phi\|_2 \le C \|Q^u_h  u - u_h\| \label{elliptic-reg}.
\end{align}

\begin{theorem}
If $A\in W^{2,\infty}_{\Tc_h}$ and \eqref{elliptic-reg} holds, 
there exists a constant $C$ independent of $h$ such that
\begin{align}
\|Q^u_h  u- u_h \| \leq Ch^2\left(\|\s\|_2 +\|\g\|_2\right). \label{superconv}
	\end{align}
\end{theorem}

\begin{proof}
We present the proof for the MSMFE-1 method and note that the proof
for the MSMFE-0 method can be obtained by omitting the terms arising
due to the quadrature error $\delta(\cdot,\cdot)$.  We rewrite
the error equation \eqref{error-P1-eq1} as 
\begin{align*}
& (A(\Pi\s- \s_h),\t)_Q + (Q^u_h  u-u_h,\dvr\t) \\
& \qquad = (A(\Pi\s-\s),\t) -\theta(A\Pi\s,\t)
- (\g, \t)+(\g_h,\t)_Q + \gnp[g-\Pc_0 g]{\t n}_{\Gamma_D}.
\end{align*}
and choose $\t =\Pi^{\RT}A^{-1}\epsilon(\phi)$ to obtain
\begin{align}
\|Q^u_h  u-u_h\|^2 = &-(A(\Pi\s- \s_h),\Pi^{\RT}A^{-1}\epsilon(\phi))_Q  
+ (A(\Pi\s-\s),\Pi^{\RT}A^{-1}\epsilon(\phi))
-\theta(A\Pi\s,\Pi^{\RT}A^{-1}\epsilon(\phi)) \nonumber \\
& - (\g, \Pi^{\RT}A^{-1}\epsilon(\phi))+(\g_h,\Pi^{\RT}A^{-1}\epsilon(\phi))_Q.
\label{super-ineq1}
\end{align}
For the second term on the right in \eqref{super-ineq1}, using
\eqref{approx-3} and \eqref{h1-continuity-rt}, we have
\begin{equation}\label{bound-term2}
|(A(\Pi\s-\s),\Pi^{\RT}A^{-1}\epsilon(\phi))| \le C h^2 \|\s\|_2\|\phi\|_2.
\end{equation}
The third term on the right in \eqref{super-ineq1} is bounded using 
\eqref{h2-theta}, \eqref{h1-continuity-bdm} and \eqref{h1-continuity-rt}:
\begin{align}
|\theta(A\Pi\s,\Pi^{\RT}A^{-1}\epsilon(\phi))|  \leq C\sum_{E\in \Tc_h} h^2 \|A\Pi\s\|_{2,E}\|\Pi^{\RT}A^{-1}\epsilon(\phi)\|_{1,E} \leq Ch^2\|\s\|_{2}\|\phi\|_2.\label{super-ineq2}
\end{align}
The first term on the right in \eqref{super-ineq1} can be manipulated as follows:
\begin{align}
& (A(\Pi\s- \s_h),\Pi^{\RT}A^{-1}\epsilon(\phi))_{Q,E}\nonumber \\
&\qquad 
= ((A-\bar A)(\Pi\s-\s_h),\Pi^{\RT}A^{-1}\epsilon(\phi))_{Q,E} 
+ (\bar A(\Pi\s-\s_h),\Pi^{\RT}(A^{-1}-\bar A^{-1})\epsilon(\phi))_{Q,E} 
\nonumber \\
&\qquad
 + (\bar A(\Pi\s-\s_h),\Pi^{\RT}\bar A^{-1}(\epsilon(\phi)-\epsilon(\phi_1)))_{Q,E}+ (\bar A(\Pi\s-\s_h),\Pi^{\RT}\bar A^{-1}\epsilon(\phi_1))_{Q,E}
\equiv \sum_{k=1}^4 I_k, \label{super-ineq4}
\end{align}
where $\bar A$ is the mean value of $A$ on $E$ and $\phi_1$ is a 
linear approximation of $\phi$ such that, see \cite{ciarlet2002finite},
\begin{align}
\|\phi-\phi_1\|_E\leq Ch^2\|\phi\|_{2,E}, 
\qquad \|\phi-\phi_1\|_{1,E}\leq Ch\|\phi\|_{2,E}. \label{lin-approx}
\end{align}
Using \eqref{approx-1}, \eqref{lin-approx}, and \eqref{h1-continuity-rt}, we
have
\begin{equation}\label{super-ineq5}
|I_1| + |I_2| + |I_3| \le C h \|\Pi\s-\s_h\|_E\|\phi\|_{2,E}.
\end{equation}
For the last term on the right in \eqref{super-ineq4}, we first note that
for a constant tensor $\t_0$, $\hat\t_0 = J_E \t_0 \DF_E^{-T} \in \Xh^{\RT}(\Eh)$,
so using \eqref{Pi-Piola} we have
\begin{equation}\label{Pi-const}
\Pi^{\RT}\t_0 = \frac{1}{J_E} \hat\Pi^{\RT}\hat\t_0 \DF_E^T
= \frac{1}{J_E} \hat\t_0 \DF_E^T = \t_0.
\end{equation}
Therefore 
\begin{equation}\label{I4}
I_4 = (\Pi\s-\s_h,\epsilon(\phi_1))_{Q,E} = 
(\Pi\s-\s_h,\nabla\phi_1)_{Q,E} - (\Pi\s-\s_h,\skew(\nabla\phi_1))_{Q,E}
\end{equation}
For the second term on the right in \eqref{I4} we write
\begin{align}
(\Pi\s-\s_h,\skew(\nabla\phi_1))_{Q,E} 
& = (\Pi\s-\s_h,\skew(\nabla\phi_1) - Q_h^{\gamma}\skew(\nabla\phi_1))_{Q,E}
+ (\Pi\s-\s_h,Q_h^{\gamma}\skew(\nabla\phi_1))_{Q,E} \nonumber \\
& \le C h \|\Pi\s-\s_h\|_E\|\phi\|_{2,E} 
+ |(\Pi\s-\s_h,Q_h^{\gamma}\skew(\nabla\phi_1))_{Q,E}|, \label{I4-1}
\end{align}
using \eqref{approx-2} for the inequality. For the last term above, 
using the error equation
\eqref{P1-error-eq-3}, we write
\begin{align}
(\Pi\s-\s_h,Q_h^{\gamma}\skew(\nabla\phi_1))_{Q,E} & = 
(\Pi\s,Q_h^{\gamma}\skew(\nabla\phi_1))_{Q,E} 
- (\Pi\s,Q_h^{\gamma}\skew(\nabla\phi_1))_{E} \nonumber \\
& \quad +
(\Pi\s-\s,Q_h^{\gamma}\skew(\nabla\phi_1))_{E} 
\le C h^2 \|\s\|_{2,E}\|\phi\|_{2,E}, \label{I4-2}
\end{align}
using \eqref{h2-theta} and \eqref{approx-3} for the inequality. 

We next bound the first term on the right in \eqref{I4}. Using that
$\nabla \phi_1 = \hat{\nabla} \hat{\phi}_1 \DF^{-1}$, we write
\begin{equation}\label{I4-3}
(\Pi\s-\s_h,\nabla\phi_1)_{Q,E} = 
(\hat{\Pi}\sh-\sh_h,\hat{\nabla}\hat{\phi}_1)_{\hat{Q},\Eh}.
\end{equation}
We note that $\hat{\phi}_1$ is bilinear. 	
Let $\tilde{\phi}_1$ be the linear part of $\hat{\phi}_1$. Then we have
	\begin{align}
	(\hat{\Pi}\sh-\sh_h,\hat{\nabla}\hat{\phi}_1)_{\hat{Q},\Eh} 
= (\hat{\Pi}\sh-\sh_h,\hat{\nabla}(\hat{\phi}_1-\tilde{\phi}_1))_{\hat{Q},\Eh} 
+ (\hat{\Pi}\sh-\sh_h,\hat{\nabla}\tilde{\phi}_1)_{\hat{Q},\Eh}. \label{I4-4}
	\end{align}
It follows from \eqref{mapping} that
$[\hat{\nabla}(\hat{\phi}_1-\tilde{\phi}_1)]_i 
= ((\r_{34}-\r_{21})\cdot [\nabla \phi_1\circ F_E]_i)\begin{pmatrix} 
\hat{y}\\ \hat{x} \end{pmatrix}$, $i=1,2$. Hence, \eqref{h2-parall} implies
\begin{align}
|(\hat{\Pi}\sh-\sh_h,\hat{\nabla}(\hat{\phi}_1-\tilde{\phi}_1))_{\hat{Q},\Eh}|
\leq Ch^2\|\hat{\Pi}\sh-\sh_h\|_{\Eh}\|\nabla\phi\circ F_E\|_{\Eh} 
\leq Ch\|\Pi\s-\s_h\|_{E}\|\phi\|_{1,E}, \label{super-ineq8}
\end{align}
where we used \eqref{scaling-1} in the last inequality. For the
last term in \eqref{I4-4}, 
using \eqref{lemma-orth-quadrature} and the exactness of the quadrature 
rule for linear functions, we obtain
\begin{align}
(\hat{\Pi}\sh-\sh_h,\hat{\nabla}\tilde{\phi}_1)_{\hat{Q},\Eh} &=  (\hat{\Pi}^{\RT}(\hat{\Pi}\sh-\sh_h),\hat{\nabla}\tilde{\phi}_1)_{\hat{Q},\Eh} = (\hat{\Pi}^{\RT}(\hat{\Pi}\sh-\sh_h),\hat{\nabla}\tilde{\phi}_1)_{\Eh} \nonumber \\
&= (\hat{\Pi}^{\RT}(\hat{\Pi}\sh-\sh_h),\hat{\nabla}(\tilde{\phi}_1-\hat{\phi}_1))_{\Eh} + (\hat{\Pi}^{\RT}(\hat{\Pi}\sh-\sh_h),\hat{\nabla}\hat{\phi}_1)_{\Eh}.\label{super-ineq9}
\end{align}
We bound the first term on the right in  \eqref{super-ineq9} similarly to
\eqref{super-ineq8}:
\begin{align}
|(\hat{\Pi}^{\RT}(\hat{\Pi}\sh-\sh_h),\hat{\nabla}(\tilde{\phi}_1-\hat{\phi}_1))_{\Eh}| 
\leq Ch^2\|\hat{\Pi}\sh-\sh_h\|_{\Eh}\|\nabla\phi\circ F_E\|_{\Eh} 
\leq Ch\|\Pi\s-\s_h\|_{E}\|\phi\|_{1,E}. \label{super-ineq10}
\end{align}
Combining \eqref{super-ineq4}--\eqref{super-ineq10} and summing over the 
elements, we obtain
	\begin{align}
|(A(\Pi\s- \s_h),\Pi^{\RT}A^{-1}\epsilon(\phi))_{Q}|
\leq  C(h\|\Pi\s-\s_h\| + h^2 \|\s\|_2)\|\phi\|_{2} 
+ \Big|\sum_{E\in\Tc_h}(\Pi^{\RT}(\Pi\s-\s_h),\nabla\phi_1)_{E}\Big|. 
\label{super-ineq11}
\end{align}
For the last term above, noting that integration by parts, \eqref{error-P1-eq9-1},
\eqref{rt-operator1}, $\phi=0$ on $\Gamma_D$, 
and $(\Pi\s-\s_h)n =0$ on $\Gamma_N$ imply 
$ \sum_{E\in \Tc_h}(\Pi^{\RT}(\Pi\s-\s_h),\nabla\phi)_{E} = 0$, we have 
\begin{align}
\Big|\sum_{E\in\Tc_h}(\Pi^{\RT}(\Pi\s-\s_h),\nabla\phi_1)_{E}\Big| 
= \Big|\sum_{E\in\Tc_h}(\Pi^{\RT}(\Pi\s-\s_h),\nabla(\phi_1-\phi))_{E}\Big| 
&\leq Ch\|\Pi\s-\s_h\|\|\phi\|_2. \label{super-ineq12}
\end{align}

It is left to bound the last two terms on the right in \eqref{super-ineq1}.
We rewrite them as follows:
\begin{align}
& - (\g, \Pi^{\RT}A^{-1}\epsilon(\phi))+(\g_h,\Pi^{\RT}A^{-1}\epsilon(\phi))_Q
\nonumber \\
& \qquad 
= -\delta(\Pi^{\RT}A^{-1}\epsilon(\phi),Q^{\g}_h \g) 
	-(\g-Q^{\g}_h \g, \Pi^{\RT}A^{-1}\epsilon(\phi)) 
+ (\g_h - Q^{\g}_h \g,\Pi^{\RT}A^{-1}\epsilon(\phi))_Q. \label{super-ineq16}
\end{align}
For the first term on the right-hand side we use 
\eqref{h2-delta}, \eqref{h1-continuity-rt}, and \eqref{h1-continuity-l2}:
\begin{align}
|\delta( \Pi^{\RT}A^{-1}\epsilon(\phi),Q^{\g}_h \g)| 
\leq C\sum_{E\in \Tc_h} h^2 \|\Pi^{\RT}A^{-1}\epsilon(\phi)\|_{1,E}
\|Q^{\g}_h \g\|_{2,E} 
\leq C h^2 \|\phi\|_{2}\|\g\|_{2}. \label{super-ineq17}
\end{align}
The second term on the right in \eqref{super-ineq16} is
bounded using the symmetry of $A^{-1}\epsilon(\phi)$,
\eqref{approx-2} and \eqref{approx-4}:
\begin{align}
|(\g-Q^{\g}_h \g,\Pi^{\RT}A^{-1}\epsilon(\phi))| 
= |(\g-Q^{\g}_h \g, \Pi^{\RT}A^{-1}\epsilon(\phi)-A^{-1}\epsilon(\phi))|  
\leq Ch^2\|\g\|_1\|\phi\|_2. \label{super-ineq18}
\end{align}
For the last term in \eqref{super-ineq16} we have
\begin{align}
(\g_h-Q_h^{\g}\g,\Pi^{\RT}A^{-1}\epsilon(\phi))_{Q} 
& = (\g_h-Q_h^{\g}\g, \Pi^{\RT}(A^{-1}-\bar A^{-1})\epsilon(\phi))_{Q}
+(\g_h-Q_h^{\g}\g, \Pi^{\RT}\bar A^{-1}(\epsilon(\phi)-\epsilon(\phi_1)))_{Q}
\nonumber \\
&\quad  + (\g_h-Q_h^{\g} \g, \Pi^{\RT}\bar A^{-1}\epsilon(\phi_1))_{Q}.
\label{super-ineq19}
\end{align}
We bound the first two terms on the right in \eqref{super-ineq19} 
similarly to $I_2$ and $I_3$ in \eqref{super-ineq5}:
\begin{multline}
|(\g_h-Q_h^{\g}\g, \Pi^{\RT}(A^{-1}-\bar A^{-1})\epsilon(\phi))_{Q} 
+(\g_h-Q_h^{\g}\g,\Pi^{\RT}\bar A^{-1}(\epsilon(\phi)-\epsilon(\phi_1)))_{Q}| 
\leq Ch\|\g_h-Q_h^{\g}\g\|\|\phi\|_{2}.
\label{super-ineq20}
\end{multline}
For the last term in \eqref{super-ineq19}, using \eqref{Pi-const} and the symmetry
of $\bar A^{-1}\epsilon(\phi_1)$, we have
\begin{align}
(\g_h-Q_h^{\g}\g,\Pi^{\RT} \bar A^{-1}\epsilon(\phi_1))_{Q} 
= (\g_h-Q_h^{\g}\g, \bar A^{-1}\epsilon(\phi_1))_{Q} =  0. \label{super-ineq21}
\end{align}
The assertion of the theorem follows from combining 
\eqref{super-ineq1}--\eqref{super-ineq21} 
and using \eqref{error-P1-final1}.
\end{proof}

\section{Numerical results}

In this section we present numerical results that verify the
theoretical results from the previous sections. We used deal.II finite
element library \cite{dealii} for the implementation of the method.
We consider a homogeneous and isotropic body,
$$ A\sigma = \frac{1}{2\mu} \left( \sigma 
- \frac{\lambda}{2\mu + 2\lambda}\operatorname{tr}(\sigma)I \right), 
$$
where $I$ is the $2 \times 2$ identity matrix and $\mu > 0$, $\lambda > -\mu$
are the Lam\'{e} coefficients. We consider $\Omega = (0,1)^2$ and the 
elasticity problem \eqref{elast-1}--\eqref{elast-2}
with Dirichlet boundary conditions and exact solution
\cite{arnold2015mixed}
$$
u _0 = \begin{pmatrix} \cos(\pi x)\sin(2\pi y) \\ \cos(\pi y)\sin(\pi x) \end{pmatrix}.
$$
The Lam\`{e} coefficients are chosen as $\lambda=123,\, \mu=79.3$. 

We study the convergence of the MSMFE-1 method on three different
types of grids. For the first test, we use the sequence of square
meshes generated by sequential uniform refinement of an initial mesh
with characteristic size $h=1/2$, see Figure~\ref{fig:1}.  For the
second test, an initial general quadrilateral grid is used, and a
sequence of meshes is obtained by sequential splitting of each element
into four. This refinement procedure produces $h^2$-parallelogram
grids, see Figure~\ref{fig:2}, where the initial coarse grid is also
shown. For the third test, we consider a sequence of smooth
quadrilateral meshes. Each mesh is produced by applying a smooth map $
\mathbf{x} = \hat{\mathbf{x}} + 0.1 \sin(2\pi \hat x)\sin(2\pi \hat
y) \begin{pmatrix} 1\\1 \end{pmatrix}$ to a uniformly refined square
mesh, starting with $h=1/2$, see Figure~\ref{fig:3}. We note that the
grids in the first and third tests satisfy both the stability
condition \ref{M2} and the $h^2$-parallelogram condition
\eqref{h2-parall}. The grids in the second test satisfy
\eqref{h2-parall}, but may violate \ref{M2} along the edges of the
initial coarse grid. However, we further note that \ref{M2} is not
needed on parallelograms and the elements here are
$h^2$-parallelograms.

The computed solutions for tests 1-3 are shown in
Figures~\ref{fig:1}--\ref{fig:3}, respectively. The solutions are
similar despite the different types of grids. The highly distorted
elements in the third test do not affect the quality of the solution.
The convergence results are presented in
Tables~\ref{tab:1}--\ref{tab:3}.  We observe at least first order of
convergence for all variables, as predicted in \eqref{msmfe-error}, as
well as superconvergence of the displacement error evaluated at the
cell centers \eqref{superconv}.

\begin{figure}[ht!]
	\centering
	\begin{subfigure}[b]{0.23\textwidth}
		\includegraphics[width=\textwidth]{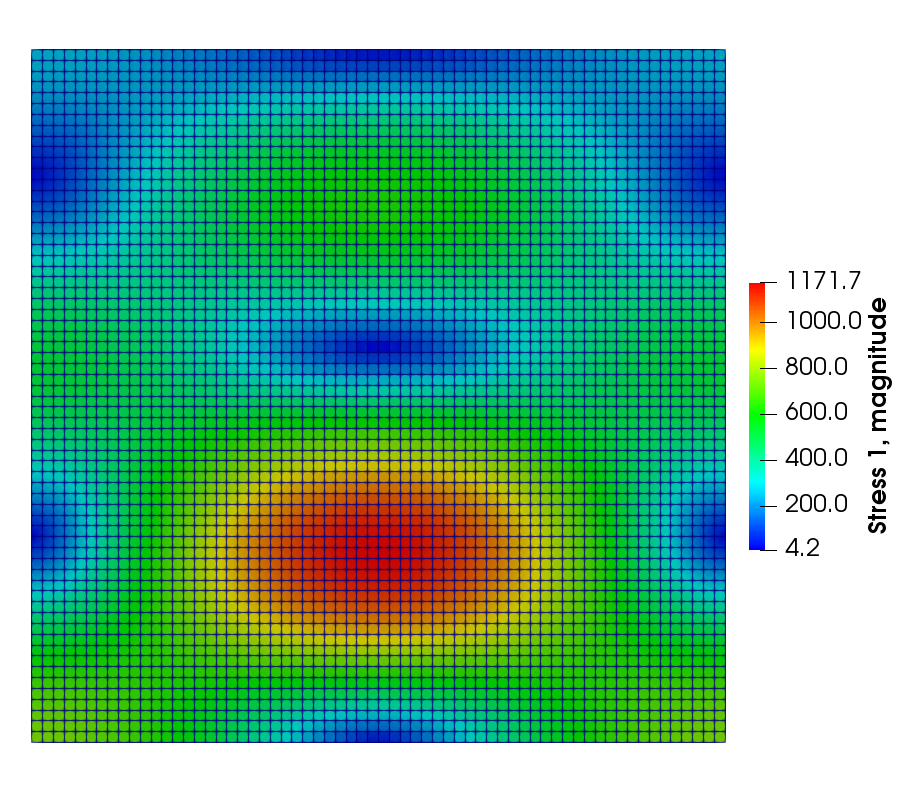}
		\caption{$x$-stress}
		\label{fig:1_1}
	\end{subfigure}
	\begin{subfigure}[b]{0.23\textwidth}
		\includegraphics[width=\textwidth]{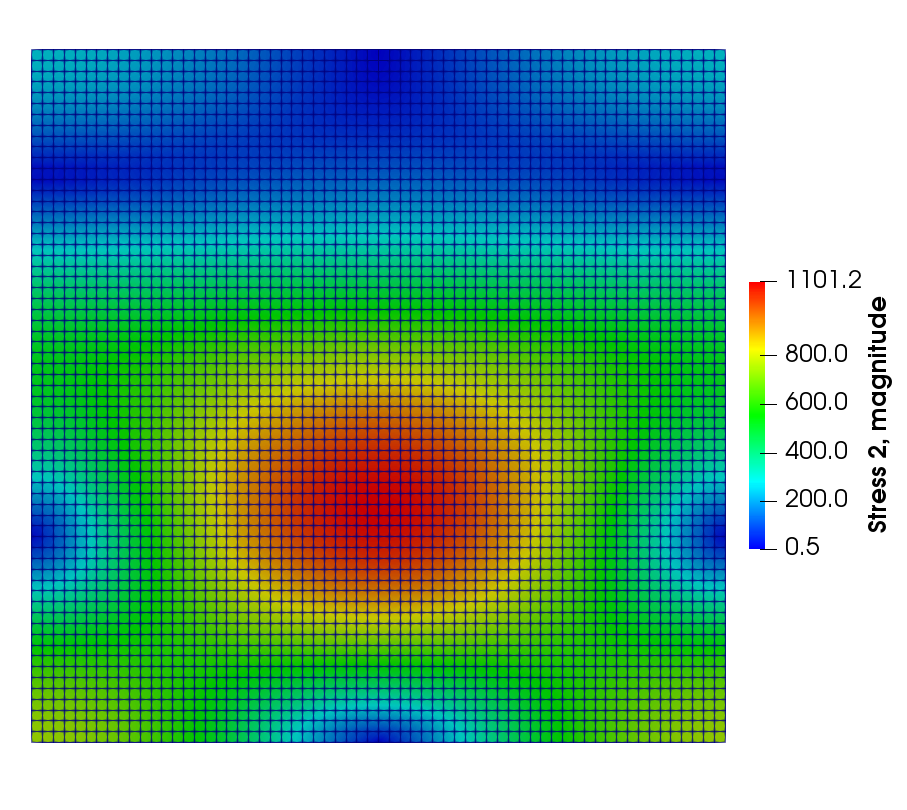}
		\caption{$y$-stress}
		\label{fig:1_2}
	\end{subfigure}
	\begin{subfigure}[b]{0.23\textwidth}
		\includegraphics[width=\textwidth]{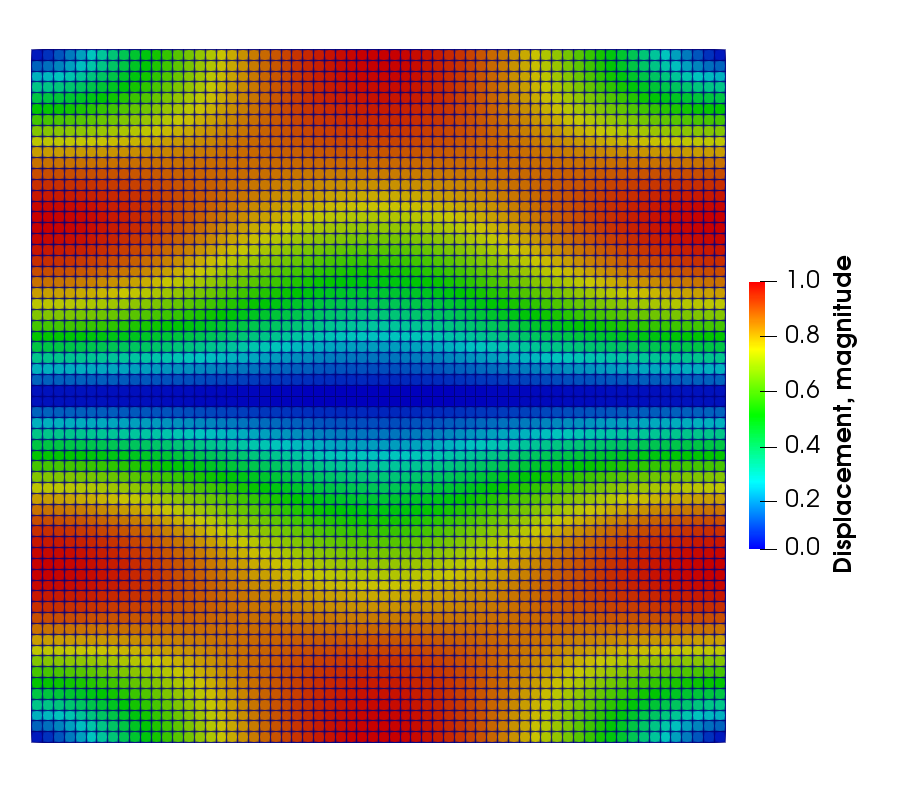}
		\caption{Displacement}
		\label{fig:1_3}
	\end{subfigure}
	\begin{subfigure}[b]{0.23\textwidth}
		\includegraphics[width=\textwidth]{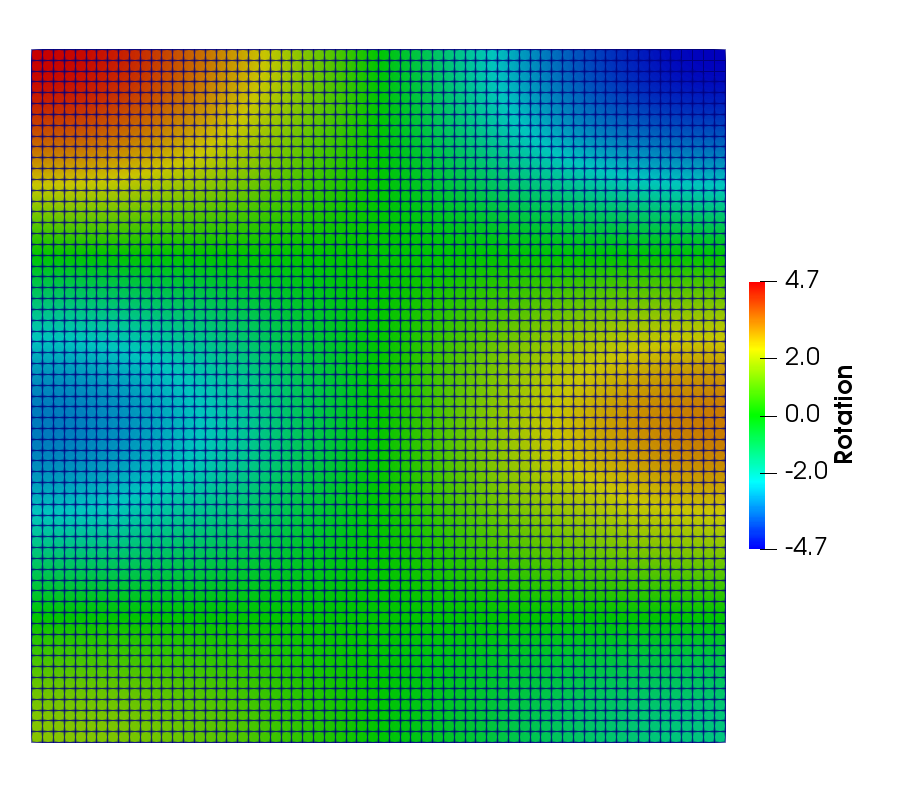}
		\caption{Rotation}
		\label{fig:1_4}
	\end{subfigure}
	\caption{Computed solution on a square mesh, $h = 1/64$.}\label{fig:1}
\end{figure}

\begin{figure}[ht!]
	\centering
	\begin{subfigure}[b]{0.23\textwidth}
		\includegraphics[width=\textwidth]{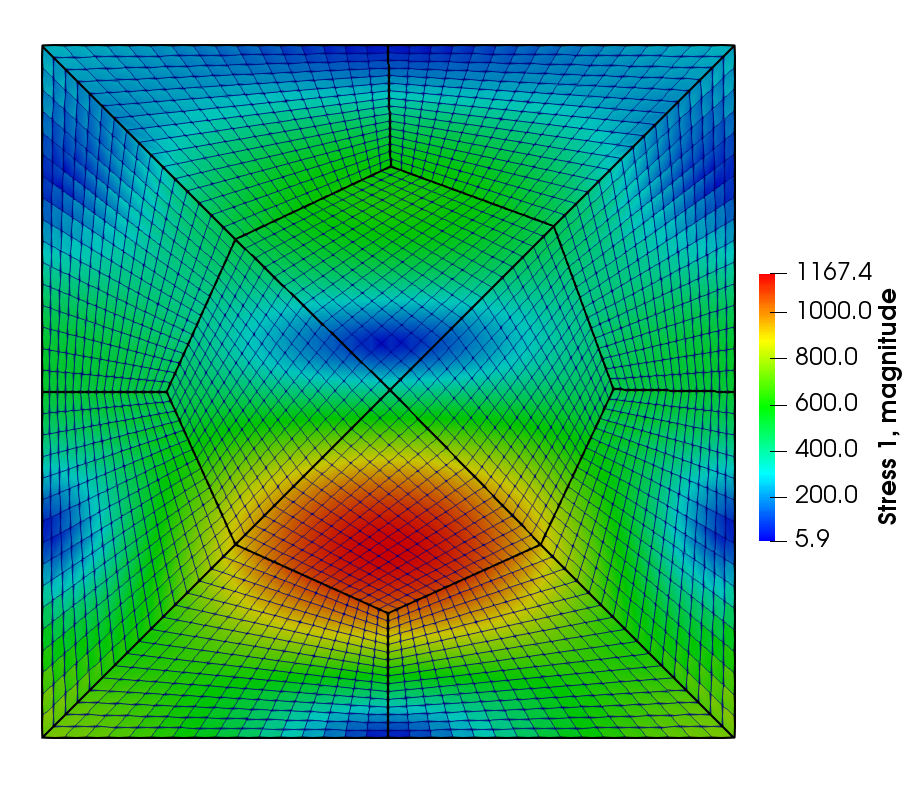}
		\caption{$x$-stress}
		\label{fig:2_1}
	\end{subfigure}
	\begin{subfigure}[b]{0.23\textwidth}
		\includegraphics[width=\textwidth]{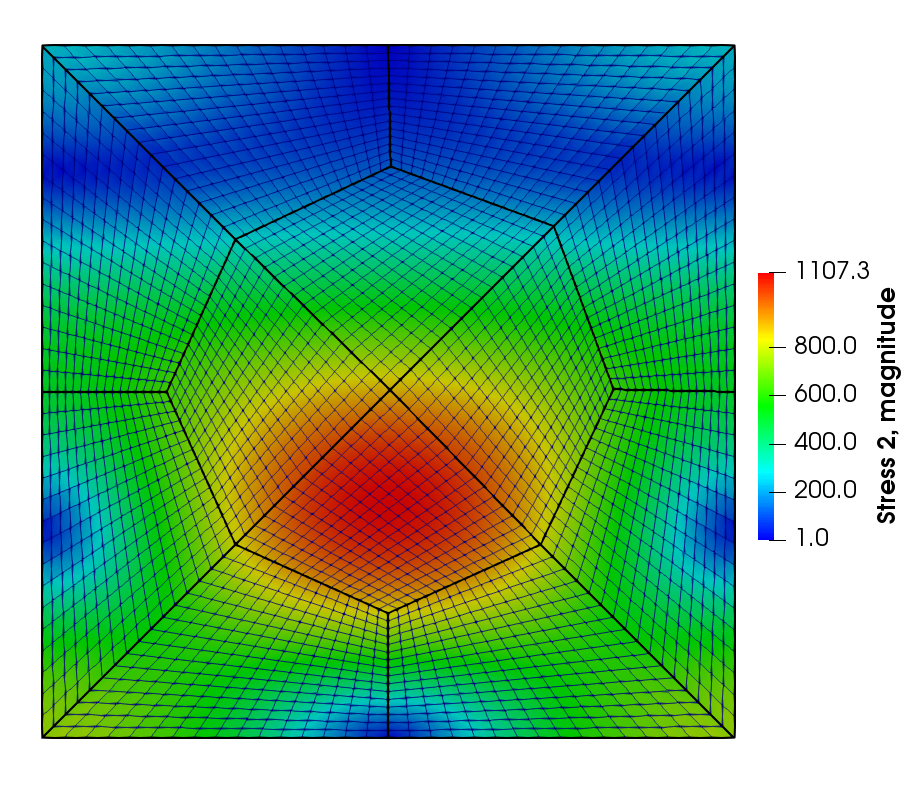}
		\caption{$y$-stress}
		\label{fig:2_2}
	\end{subfigure}
	\begin{subfigure}[b]{0.23\textwidth}
		\includegraphics[width=\textwidth]{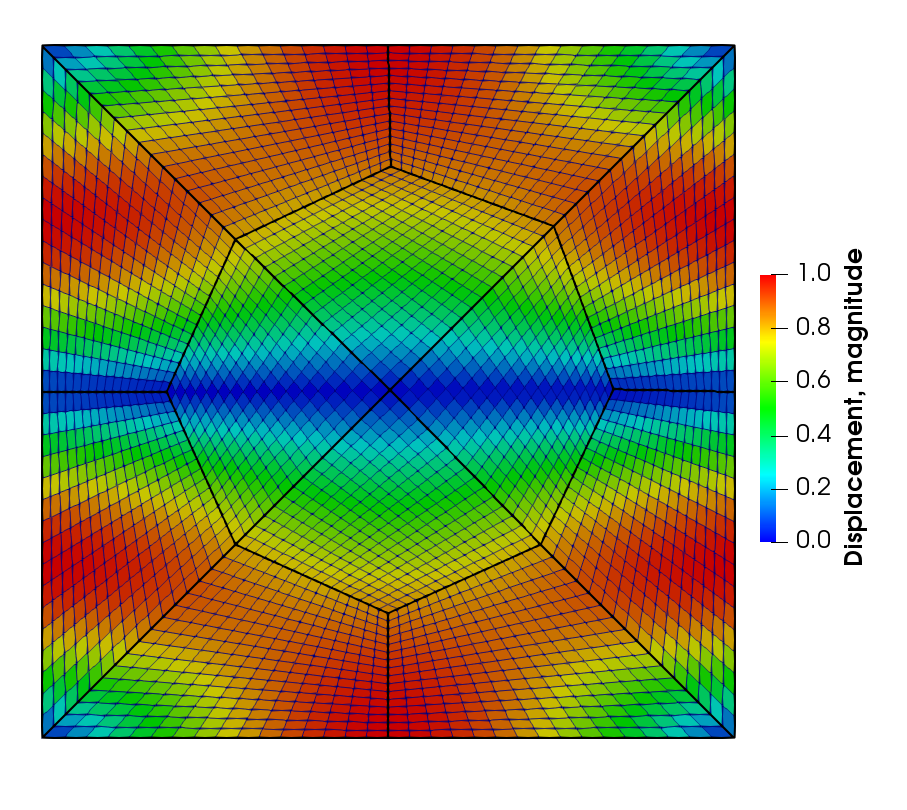}
		\caption{Displacement}
		\label{fig:2_3}
	\end{subfigure}
	\begin{subfigure}[b]{0.23\textwidth}
		\includegraphics[width=\textwidth]{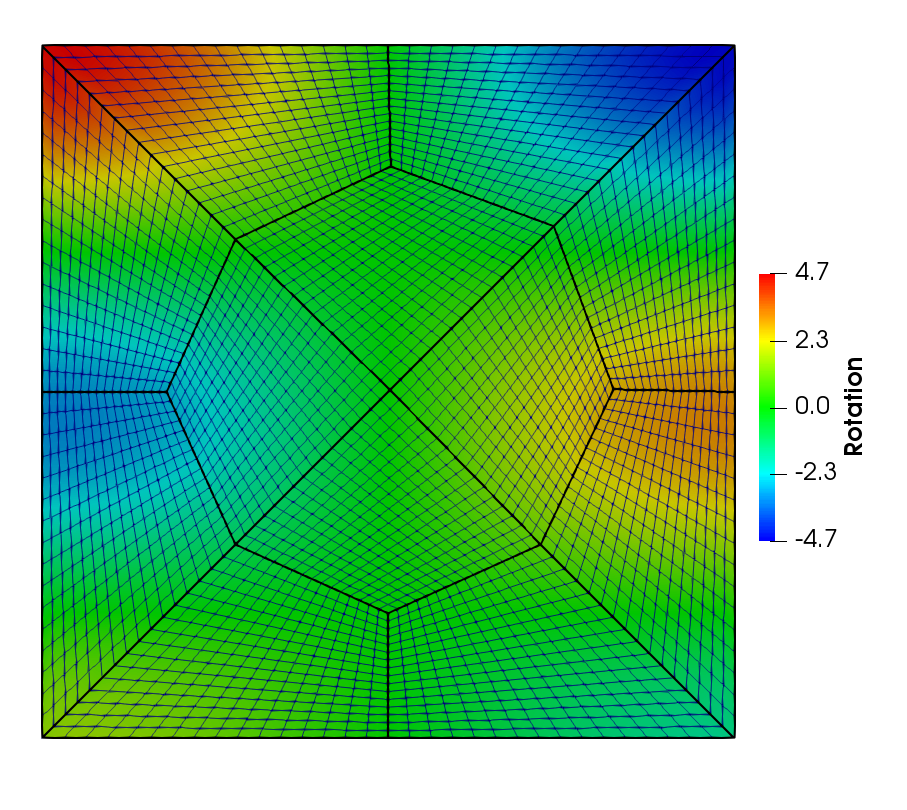}
		\caption{Rotation}
		\label{fig:2_4}
	\end{subfigure}
	\caption{Computed solution on a $h^2$-parallelogram mesh, $h = 1/32$.}\label{fig:2}
\end{figure}

\begin{figure}[ht!]
	\centering
	\begin{subfigure}[b]{0.23\textwidth}
		\includegraphics[width=\textwidth]{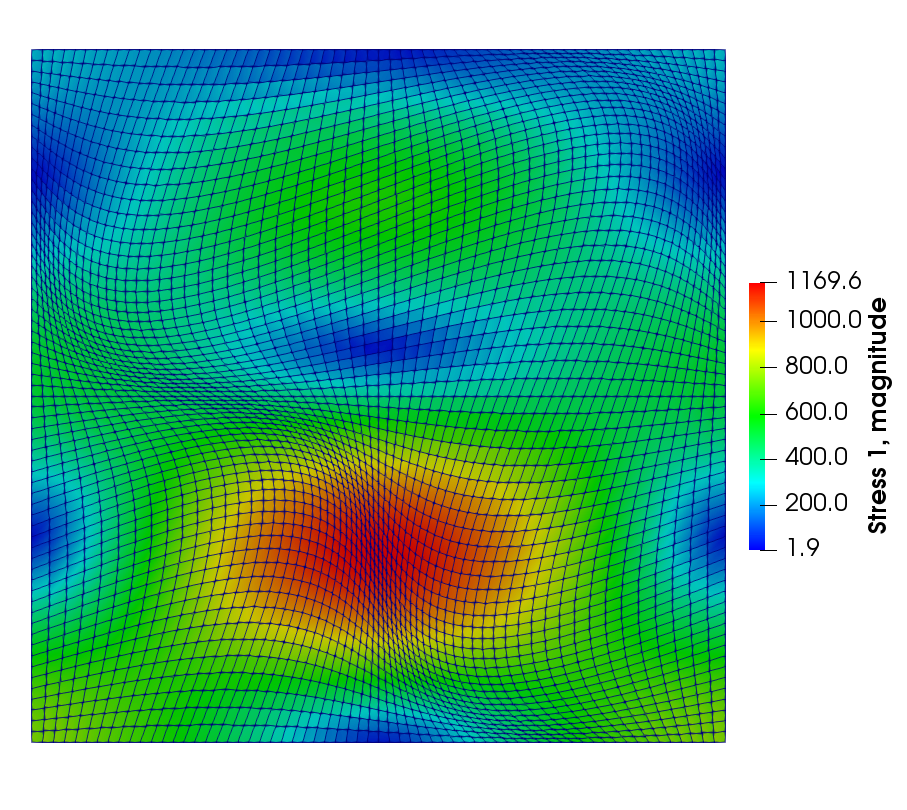}
		\caption{$x$-stress}
		\label{fig:3_1}
	\end{subfigure}
	\begin{subfigure}[b]{0.23\textwidth}
		\includegraphics[width=\textwidth]{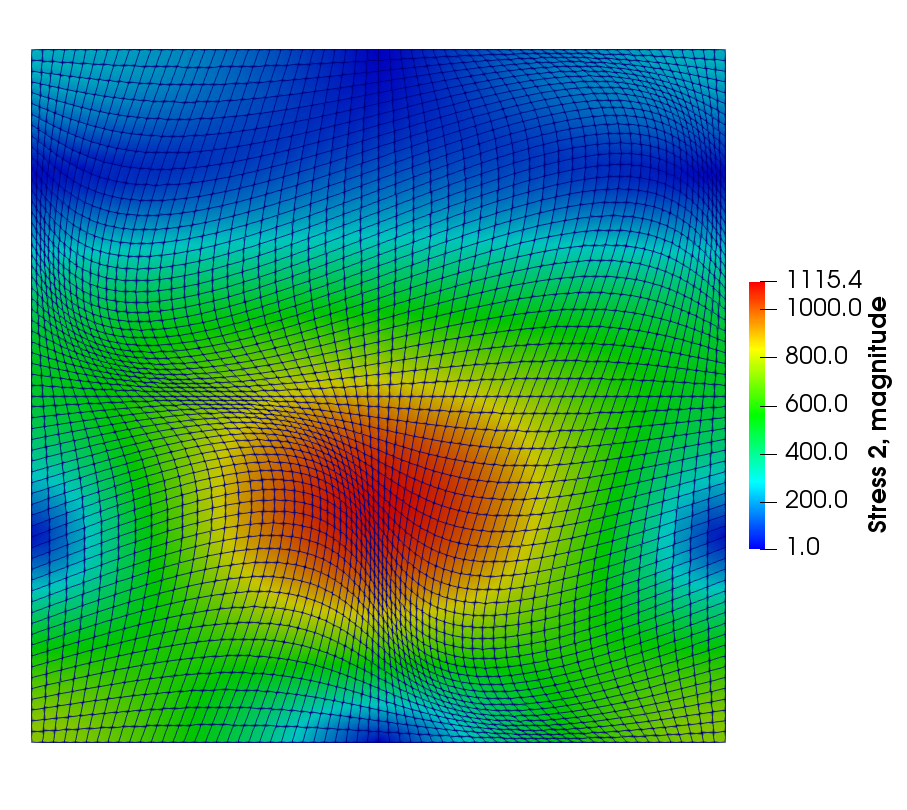}
		\caption{$y$-stress}
		\label{fig:3_2}
	\end{subfigure}
	\begin{subfigure}[b]{0.23\textwidth}
		\includegraphics[width=\textwidth]{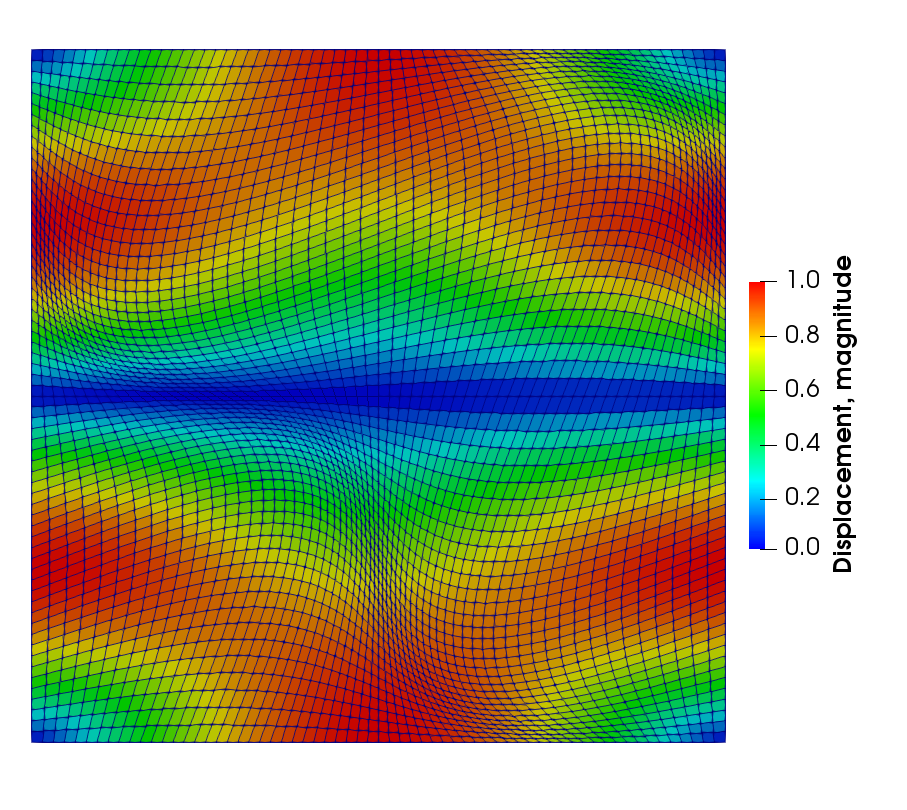}
		\caption{Displacement}
		\label{fig:3_3}
	\end{subfigure}
	\begin{subfigure}[b]{0.23\textwidth}
		\includegraphics[width=\textwidth]{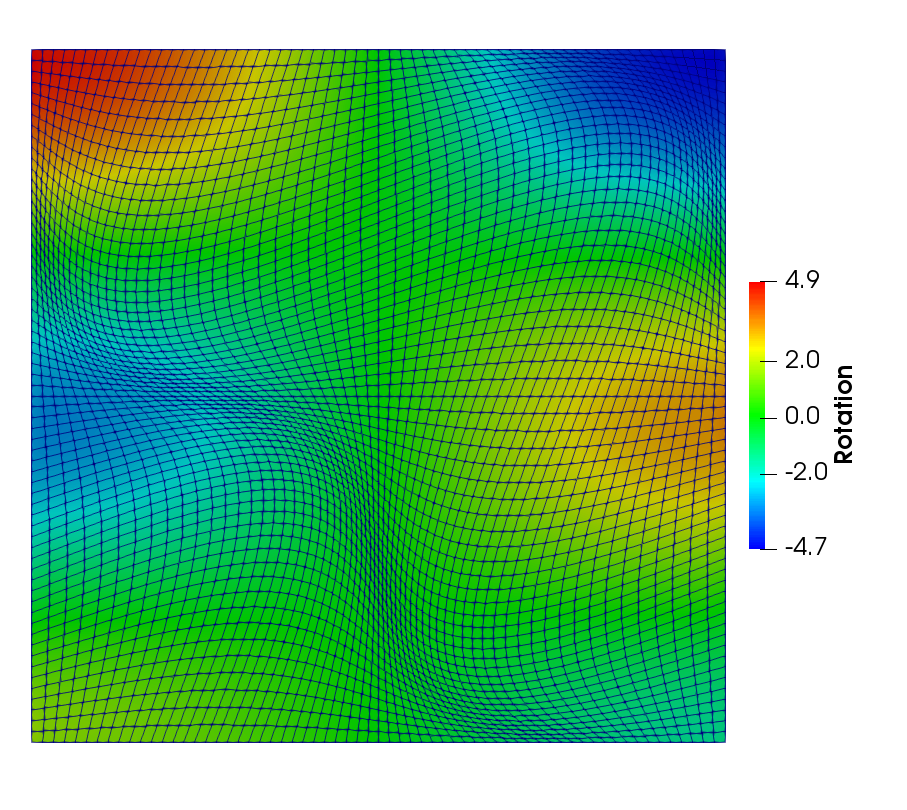}
		\caption{Rotation}
		\label{fig:3_4}
	\end{subfigure}
	\caption{Computed solution on a smooth quadrilateral mesh, $h=1/64$.}\label{fig:3}
\end{figure}

\begin{table}[ht!]
	\begin{center}
		\begin{tabular}{r|cc|cc|cc|cc|cc} \hline
			& 
			\multicolumn{2}{c|}{$ \|\sigma - \sigma_h\| $} & 
			\multicolumn{2}{c|}{$ \|\dvr(\sigma - \sigma_h)\|$} & 
			\multicolumn{2}{c|}{$  \|u - u_h\| $} & 
			\multicolumn{2}{c|}{$ \|Q^u_hu - u_h\| $} & 
			\multicolumn{2}{c}{$ \|\g - \g_h\| $}\\ 
			$h$ & error & rate & error & rate & error & rate & error & rate & error & rate \\ \hline
			1/2	&	7.61E-01	&	-	&	9.73E-01	&	-	&	7.19E-01	&	-	&	4.76E-01	&	-	&	8.17E-01	&	-\\	
			1/4	&	3.74E-01	&	1.02	&	5.42E-01	&	0.84	&	4.56E-01	&	0.66	&	1.06E-01	&	2.17	&	3.91E-01	&	1.06\\	
			1/8	&	1.66E-01	&	1.17	&	2.72E-01	&	0.99	&	2.33E-01	&	0.97	&	2.76E-02	&	1.93	&	1.15E-01	&	1.77\\	
			1/16	&	7.91E-02	&	1.07	&	1.36E-01	&	1.00	&	1.17E-01	&	0.99	&	7.25E-03	&	1.94	&	3.043-02	&	1.92\\	
			1/32	&	3.90E-02	&	1.02	&	6.79E-02	&	1.00	&	5.86E-02	&	1.00	&	1.84E-03	&	1.98	&	7.75E-03	&	1.97\\	
			1/64	&	1.94E-02	&	1.01	&	3.39E-02	&	1.00	&	2.93E-02	&	1.00	&	4.62E-04	&	1.99	&	1.95E-03	&	1.99\\	\hline
		\end{tabular}
	\end{center}
	\caption{Convergence on square grids.} \label{tab:1} 
\end{table}

\begin{table}[ht!]
	\begin{center}
		\begin{tabular}{r|cc|cc|cc|cc|cc} \hline
			& 
			\multicolumn{2}{c|}{$ \|\sigma - \sigma_h\| $} & 
			\multicolumn{2}{c|}{$ \|\dvr(\sigma - \sigma_h)\|$} & 
			\multicolumn{2}{c|}{$  \|u - u_h\| $} & 
			\multicolumn{2}{c|}{$ \|Q^u_hu - u_h\| $} & 
			\multicolumn{2}{c}{$ \|\g - \g_h\| $}\\ 
			$h$ & error & rate & error & rate & error & rate & error & rate & error & rate \\ \hline
			1/3	&	5.92E-01	&	-	&	8.00E-01	&	-	&	5.35E-01	&	-	&	1.63E-01	&	-	&	5.98E-01	&	-\\	
			1/6	&	2.78E-01	&	1.09	&	4.06E-01	&	0.98	&	3.11E-01	&	0.78	&	1.05E-01	&	0.63	&	3.38E-01	&	0.82\\	
			1/12	&	1.37E-01	&	1.02	&	2.03E-01	&	1.00	&	1.58E-01	&	0.98	&	2.95E-02	&	1.84	&	1.38E-01	&	1.30\\	
			1/24	&	6.93E-02	&	0.98	&	1.01E-01	&	1.00	&	7.90E-02	&	1.00	&	8.04E-03	&	1.87	&	4.87E-02	&	1.50\\	
			1/48	&	3.50E-02	&	0.99	&	5.07E-02	&	1.00	&	3.95E-02	&	1.00	&	2.08E-03	&	1.95	&	1.66E-02	&	1.55\\	
			1/96	&	1.76E-02	&	0.99	&	2.53E-02	&	1.00	&	1.97E-02	&	1.00	&	5.26E-04	&	1.98	&	5.67E-03	&	1.55\\	\hline
		\end{tabular}
	\end{center}
	\caption{Convergence on $h^2$-parallelogram grids.} \label{tab:2} 
\end{table}

\begin{table}[ht!]
	\begin{center}
		\begin{tabular}{r|cc|cc|cc|cc|cc} \hline
			 & 
			\multicolumn{2}{c|}{$ \|\sigma - \sigma_h\| $} & 
			\multicolumn{2}{c|}{$ \|\dvr(\sigma - \sigma_h)\|$} & 
			\multicolumn{2}{c|}{$  \|u - u_h\| $} & 
			\multicolumn{2}{c|}{$ \|Q^u_hu - u_h\| $} & 
			\multicolumn{2}{c}{$ \|\g - \g_h\| $}\\ 
			$h$ & error & rate & error & rate & error & rate & error & rate & error & rate \\ \hline
			1/4	&	4.27E-01	&	-	&	6.22E-01	&	-	&	4.71E-01	&	-	&	1.64E-01	&	-	&	4.53E-01	&	-\\	
			1/8	&	2.22E-01	&	0.94	&	3.46E-01	&	0.85	&	2.68E-01	&	0.81	&	7.09E-02	&	1.21	&	2.14E-01	&	1.08\\	
			1/16	&	1.12E-01	&	0.99	&	1.78E-01	&	0.96	&	1.37E-01	&	0.97	&	2.51E-02	&	1.50	&	9.29E-02	&	1.21\\	
			1/32	&	5.61E-02	&	1.00	&	9.00E-02	&	0.99	&	6.84E-02	&	1.00	&	7.35E-03	&	1.77	&	3.21E-02	&	1.53\\	
			1/64	&	2.81E-02	&	1.00	&	4.51E-02	&	1.00	&	3.42E-02	&	1.00	&	1.94E-03	&	1.92	&	1.04E-02	&	1.63\\	
			1/128	&	1.40E-02	&	1.00	&	2.26E-02	&	1.00	&	1.71E-02	&	1.00	&	4.93E-04	&	1.98	&	3.41E-03	&	1.61\\	\hline
		\end{tabular}
	\end{center}
	\caption{Convergence on smooth quadrilateral grids.} \label{tab:3} 
\end{table}

\bibliographystyle{abbrv}
\bibliography{msmfe-quads}

\end{document}

%% file: fig-M.tex
\begin{figure}	
	\setlength{\unitlength}{0.8mm}
	\scalebox{.8}{\begin{picture}(50,65)(-65,-20)
	\thicklines
	\Line(-18,0)(26,-1)
	\Line(-18,0)(-15,48)
	\Line(26,-1)(29,50)
	\Line(-15,48)(29,50)
	
	\put(-18,0){\circle{2}}
	\put(26,-1){\circle{2}}
	\put(-15,48){\circle{2}}
	\put(29,50){\circle{2}}	
	
	\put(-16.5,2){\text{{$\r_1$}}}
	\put(21.5,1.5){\text{$\r_2$}}
	\put(23,46){\text{$\r_3$}}
	\put(-13.5,45){\text{$\r_4$}}
	
	\Line(-2,-2.5)(2,1.5)
	\Line(2,-2.5)(-2,1.5)
	
	\Line(25.5,20)(29.5,24)
	\Line(29.5,20)(25.5,24)
	
	\Line(4,47)(8,51)
	\Line(8,47)(4,51)
	
	\Line(-18.5,21)(-14.5,25)
	\Line(-14.5,21)(-18.5,25)
	
	\put(-1,-4.5){\text{$q_1$}}
	\put(30,22){\text{$q_2$}}
	\put(4.5,52.5){\text{$q_3$}}
	\put(-23,23){\text{$q_4$}}
	
	\Line(52,0)(101,-1)
	\Line(52,0)(55,50)
	\Line(101,-1)(99,48)
	\Line(55,50)(99,48)
	
	\put(52,0){\circle{2}}
	\put(101,-1){\circle{2}}
	\put(55,50){\circle{2}}
	\put(99,48){\circle{2}}	
	
	\put(53.5,2){\text{{$\r_1$}}}
	\put(95.5,1.5){\text{$\r_2$}}
	\put(93.5,45){\text{$\r_3$}}
	\put(56,46){\text{$\r_4$}}
	
	
	\Line(98,20.5)(102,24.5)
	\Line(102,20.5)(98,24.5)
	
	\Line(75,47)(79,51)
	\Line(79,47)(75,51)
	
	\Line(51.5,22)(55.5,26)
	\Line(55.5,22)(51.5,26)
	
	\put(102,22){\text{$q_2$}}
	\put(75.5,52.5){\text{$q_3$}}
	\put(47,23){\text{$q_4$}}
	
	\put(73,-10){\text{$\Gamma_N$}}
	
	\Line(123,8)(143,9)
	\Line(121,28)(142,31)
	\Line(121,28)(123,8)
	\Line(143,9)(142,31)
	\Line(143,9)(164,8)
	\Line(142,31)(163,33)
	\Line(143,9)(142,31)
	\Line(164,8)(163,33)
	\Line(163,33)(184,32)
	\Line(164,8)(186,6)
	\Line(163,33)(164,8)
	\Line(184,32)(186,6)
	\Line(142,31)(140,53)
	\Line(163,33)(164,54)
	\Line(140,53)(164,54)
	\Line(143,9)(141,-15)
	\Line(164,8)(166,-12)
	\Line(141,-15)(166,-12)

	\put(151,19){\text{$E$}}
	\put(151,-5){\text{$\tilde{E}$}}
	
	\put(123,8){\circle{2}}
	\put(143,9){\circle{2}}
	\put(121,28){\circle{2}}
	\put(142,31){\circle{2}}
	\put(164,8){\circle{2}}
	\put(163,33){\circle{2}}
	\put(184,32){\circle{2}}
	\put(186,6){\circle{2}}
	\put(140,53){\circle{2}}
	\put(164,54){\circle{2}}
	\put(141,-15){\circle{2}}
	\put(166,-12){\circle{2}}
	
	\put(138,11){\text{$\r_1$}}
	\put(165,9.5){\text{$\r_2$}}
	\put(164.5,30){\text{$\r_3$}}
	\put(137.5,27.5){\text{$\r_4$}}
	\put(135,-15){\text{$\tilde{\r}_1$}}
	\put(168,-13.5){\text{$\tilde{\r}_2$}}	
	
	\Line(153.5,7.5)(155.5,9.5)
	\Line(155.5,7.5)(153.5,9.5)
	
	\Line(152,31)(154,33)
	\Line(154,31)(152,33)
	
	\Line(162.5,20)(164.5,22)
	\Line(164.5,20)(162.5,22)
	
	\Line(141.5,20)(143.5,22)
	\Line(143.5,20)(141.5,22)
	
	\put(152.5,3.5){\text{$q_1$}}
	\put(166,21){\text{$q_2$}}
	\put(152,35.5){\text{$q_3$}}
	\put(136.5,21){\text{$q_4$}}
	
	\end{picture}
	}
	\caption{Left: interior element; middle: element with bottom edge on $\Gamma_N$; right: an interior element, surrounded by four elements. }
	\label{macroelements}	
\end{figure}